\documentclass[11pt]{amsart}
\usepackage{amsmath,amsthm,amsfonts,amscd,amssymb,eucal,latexsym,mathrsfs,stmaryrd,enumitem,bm}
\usepackage{hyperref}
\usepackage[all]{xy}
\usepackage{mathtools}
\usepackage{enumitem}  
\usepackage{bbm}
\usepackage{comment}

\newtheorem{theorem}{Theorem}[section]
\newtheorem{corollary}[theorem]{Corollary}
\newtheorem{lemma}[theorem]{Lemma}
\newtheorem{proposition}[theorem]{Proposition}

\theoremstyle{definition}
\newtheorem{definition}[theorem]{Definition}
\newtheorem{remark}[theorem]{Remark}

\newtheorem{problem}[theorem]{Problem}

\theoremstyle{plain}
\newcounter{theoremintro}
\newtheorem{theoremi}[theoremintro]{Theorem}
\newtheorem{corollaryi}[theoremintro]{Corollary}
\newtheorem{problemi}[theoremintro]{Problem}

\newcommand{\mdim}{{\rm mdim}}

\newcommand{\cW}{{\mathcal W}}

\newcommand{\cC}{{\mathcal C}}
\newcommand{\cU}{{\mathcal U}}

\newcommand{\cS}{{\mathcal S}}
\newcommand{\cZ}{{\mathcal Z}}
\newcommand{\cF}{{\mathcal F}}
\newcommand{\cV}{{\mathcal V}}

\newcommand{\sC}{{\mathscr C}}

\newcommand{\Zb}{{\mathbb Z}}

\newcommand{\Rb}{{\mathbb R}}
\newcommand{\Nb}{{\mathbb N}}

\newcommand{\eps}{\varepsilon}

\newcommand{\abs}[1]{\left|#1\right|}
\newcommand{\brak}[1]{\left(#1\right)}
\newcommand{\sqbrak}[1]{\left[#1\right]}
\newcommand{\GonX}{G \curvearrowright X}

\newcommand{\act}{\curvearrowright}
\newcommand{\lD}{\underline{D}}
\newcommand{\uD}{\overline{D}}
\newcommand{\finsub}{\subset\!\subset}
\newcommand{\ord}{\mathrm{ord}}
\newcommand{\norm}[1]{\left\|#1\right\|}

\numberwithin{equation}{section}

\renewcommand{\phi}{\varphi}

\DeclareMathOperator{\supp}{supp}
\DeclareMathOperator{\diam}{diam}

\allowdisplaybreaks

\begin{document}

\title{URP, comparison, mean dimension, and sharp shift embeddability}

\author{Petr Naryshkin}
\address{Petr Naryshkin,
Alfréd Rényi Institute of Mathematics, Budapest, Reáltanoda utca 13-15, 1053, Hungary}
\email{pnaryshkin@renyi.hu}

\begin{abstract}
For a free action $\GonX$ of an amenable group on a compact metrizable space, we study the Uniform Rokhlin Property (URP) and the conjunction of Uniform Rokhlin Property and comparison (URPC). We give several equivalent formulations of the latter and show that it passes to extensions. We introduce technical conditions called property FCSB and property FCSB in measure, both of which reduce to the marker property if $G$ is abelian. Our first main result is that for any amenable group $G$ property FCSB in measure is equivalent to URP, and for a large class of amenable groups property FCSB is equivalent to URPC. In the latter case it follows that if the action is moreover minimal then the C$^*$-crossed product $C(X) \rtimes G$ has stable rank one, satisfies the Toms-Winter conjecture, and is classifiable if $\mdim(\GonX) = 0$.

Our second main result is that if a system $\GonX$ has URPC and $\mdim(\GonX) < M/2$, then it is embeddable into the $M$-cubical shift $\brak{[0, 1]^M}^G$. Combined with the first main result, we recover the Gutman-Qiao-Tsukamoto sharp shift embeddability theorem as a special case. Notably, the proof avoids the use of either Euclidean geometry or signal analysis and directly extends the theorem to all abelian groups.

Finally, we show that if $G$ is a nonamenable group that contains a free subgroup on two generators and $\GonX$ is a topologically amenable action, then it is embeddable into $[0, 1]^G$.
\end{abstract}

\date{\today}

\maketitle

\tableofcontents

\section{Introduction}

Let $\GonX$ be a free action by homeomorphisms of an infinitely countable discrete amenable group on a compact metrizable space. An invariant called \emph{mean dimension} (denoted $\mdim(\GonX)$) for such a system was introduced by Gromov \cite{Gro99} and studied by Lindenstrauss and Weiss \cite{LinWei00}\footnote{in the case of $\Zb$. For general amenable groups see \cite{CooKri05}.}. From the onset \cite{Lin99}, the following question of \emph{shift embeddability} became one of the central problems in the field.

\begin{problem}\label{probi: main}
Let $\GonX$ be a topological dynamical system as above and let $M \in \Nb$ be a natural number. Does there exist a $G$-equivariant embedding of $X$ into the space $\brak{[0,1]^M}^G$, where $G$ acts by left shifts?
\end{problem}

The Bernoulli shift $G \act \brak{[0,1]^M}^G$ (from now on referred to as the $M$-cubical shift) has mean dimension equal to $M$. As $G$-equivariant embeddings cannot lower mean dimension, this immediately gives a negative answer for any system $\GonX$ with $\mdim(\GonX) > M$. After further study, it turned out that $M/2$ is the dividing line to the question above in the following sense. For any amenable group $G$ there exists an action $\GonX$ with $\mdim(\GonX) = M/2$ that is not embeddable into the $M$-cubical shift \cite{LinTsu14}, \cite{JinParQia22}. On the other hand, following the breakthrough papers \cite{GutLinTsu16} and \cite{GutTsu20}, Gutman, Qiao, and Tsukamoto \cite{GutQiaTsu19} obtained the following positive result.

\begin{theorem}\label{thmi: GQT}
Let $\Zb^d \act X$ be a free action with the marker property (Definition \ref{def: marker prop}) such that
\[
\mdim(\Zb^d \act X) < M/2.
\]
Then it is embeddable into the $M$-cubical shift.
\end{theorem}

The marker property was first developed by Gutman in \cite{Gut15} and \cite{Gut17}. Although it is a fairly mild assumption (for instance, it is automatic for any minimal action), there are free actions (even of $\Zb$) that do not possess it \cite{TsuTsuYos22}, \cite{Shi23}. As far as we know, Problem \ref{probi: main} is completely open for these examples. The proof of Theorem \ref{thmi: GQT} involves three main arguments that can be roughly described as follows:
\begin{enumerate}
    \item use the marker property to construct a dynamical tiling of the space such that most points (in the sense of Banach density, see Definition \ref{def: Banach density}) lie in some tile,
    \item show that the set of points that are not in the tiles is small not only in density but, in fact, in a certain (stronger) combinatorial sense, and
    \item encode a sequence of the data constructed in items (i) and (ii) using a small amount of information.
\end{enumerate}
Importantly, arguments (i) and (iii) above use the Eucledean geometry of $\Zb^d \subset \Rb^d$ and signal analysis on $\Zb^d$ respectively and therefore do not extend to action of more general amenable groups. One other feature of the results above is that they do not develop much abstract theory. Indeed, the constructions are all explicit and keep being used throughout the proof. We additionally remark that in the special case of minimal $\Zb$-actions, a new approach has been recently found by Levin \cite{Lev23}.

Another topic that has connections to mean dimension and where somewhat different dynamical tilings appear is the study of crossed product C$^*$-algebras $C(X) \rtimes G$ associated to topological dynamical systems $\GonX$ as above. After several decades of work (see \cite{TomWin13}, \cite{Ker20}, and \cite{EllNiu17} for the recent landmark papers), a number of important questions about the crossed product were reduced to a few key properties of the action $\GonX$. We focus on two such conditions: the \emph{Uniform Rokhlin Property (URP)}, introduced by Niu, which asks for existence of an Ornstein-Weiss-type castle that is large in Banach density (see Definition \ref{def: URP}), and \emph{comparison}, which allows one to upgrade the inequality between two sets in invariant measures to a combinatorial subequivalence (see Definition \ref{def: comparison}). To shorten the notation, we say that an action has \emph{URPC} if it has both URP and comparison. We summarize the known results (\cite{Niu22}, \cite{Niu21}, \cite{LiNiu20}) involving these properties.

\begin{theorem}\label{thmi: crossed prod prop}
Let $\GonX$ be a free action of an amenable group on a compact metrizable space.
\begin{itemize}
    \item It has the small boundary property if and only if it has URP and $\mdim(\GonX) = 0$.
    \item If it is minimal, has URPC, and $\mdim(\GonX) = 0$ then the crossed product $C(X) \rtimes G$ is $\cZ$-stable.
    \item  If it is minimal and has URPC then the radius of comparison of the crossed product satisfies $\mathrm{rc}(C(X) \rtimes G) \le \mdim(\GonX)/2$.
    \item If it is minimal and has URPC the the crossed product has stable rank one.
\end{itemize}
\end{theorem}

It is confirmed (see, for example, \cite{Nar24}) that URPC holds for actions of many amenable groups when the covering dimension of $X$ is finite (note that URP is automatic in this case \cite{KerSza20}). On the other hand, not much has been known in the general situation, besides for the Theorem \ref{thmi: GLT} below. 

There has been some attempts to apply the theory developed for the study of crossed products to the shift embeddability problem (see, for instance, \cite{WeiHe23}). However, these results contain neither of the three crucial arguments (i)--(iii) in the proof of Theorem \ref{thmi: GQT}: the action is assumed to be an extension of one on a finite-dimensional space (which makes constructing the tiling much easier); instead of having a small remainder set, the towers are allowed to overlap and have to cover $X$ (which increases the dimension of the cube by a multiplicative constant); and instead of encoding the tiling construction, one simply encodes the whole finite-dimensional quotient space (which increases the dimension of the cube by an additive constant). In particular, outside of certain specific circumstances \cite{LanSza23}, one never achieves the sharp bound $M/2$.

In the present paper, we first show (Theorem \ref{thmi: shift embed}) that the abstract tiling theory can, in fact, be applied to the shift embeddability problem if one proceeds in a more careful manner. We then find a common generalization (Theorem \ref{thmi: FCSB URPC}) for the tiling techniques developed in the crossed product setting (which worked for a large class of amenable groups but only on the finite-dimensional spaces) and in the shift embeddability setting (which worked only for $\Zb^d$ but on more general spaces). In other words, we completely unify these two problems in a common framework. To the best of our knowledge, this gives the most general results in either direction. 

The open sets of $X$ generate the \emph{type semigroup} equipped with the preorder given by dynamical subequivalence (Definition \ref{def: type semigroup}). We begin by setting up a machinery of calculations in the type semigroup. Most importantly, we show that the marker sets given by the marker property can be regarded as arbitrary small positive elements (Lemma \ref{Lem: mult L-free subeq}). This simple observation is the cornerstone of all our techniques going forward. We then study properties URP and URPC in depth and prove several equivalent formulations for the latter (Theorem \ref{thm: URP+comp equiv}). Our first main result is the sharp shift embeddability theorem under the assumption of URPC. We remark that this is immediately applicable to extensions (Corollary \ref{cor: URPC passes to ext}) of many finite-dimensional systems and, in contrast with the results mentioned above, gives the statement with the optimal constant $M/2$.

\begin{theoremi}\label{thmi: shift embed}
Let $\GonX$ be a free action of an amenable group on a compact metrizable space. If it has URPC and 
\[
\mdim(\GonX) < M/2
\]
then there exists a $G$-equivariant embedding of $X$ into $\brak{[0, 1]^M}^G$.
\end{theoremi}

When compared to the proof of Theorem \ref{thmi: GQT}, URP and comparison play the role of the dynamical tiling constructed in item (i) and the argument for the remainder in item (ii) respectively. For item (iii), although we follow the general philosophy of encoding using a small amount of information, the specific approach needs to be modified. Indeed, Gutman, Qiao, and Tsukamoto encode the tilings \emph{spectrally} using the Fourier transform, which is only possible on $\Zb^d$. In contrast, we do it \emph{spatially} in a more straightforward manner. The fact that a witness to property URPC can be reconstructed from a small amount of data is made precise with the notion of \emph{encoding element} in the type semigroup (see Definition \ref{def: URPC witness} and Corollary \ref{cor: encoding element recovers URPC}). We then show that for any marker set in $X$ one can find a function supported on it that encodes a whole sequence of witnesses to URPC (Lemma \ref{lem: encoding URPC 2-dim}). A slight adjustment to the usual shift embeddability argument then allows one to obtain the result above (Theorem \ref{thm: embedding}).

We remark that as there are no restrictions on the structure of the group, the core argument is not limited to actions of amenable groups. In \cite{GarGefKraNar23} it was shown that \emph{topologically amenable} (Definition \ref{def: topologically amenable}) actions of nonamenable groups often behave in highly paradoxical ways. Combined with the encoding technique described above, we can prove the following (Theorem \ref{thm: embedding nonamen}).

\begin{theoremi}\label{thmi: shift embed nonamen}
Let $G$ be a nonamenable group containing a free subgroup on two generators. If $\GonX$ is a topologically amenable action then there exists a $G$-equivariant embedding of $X$ into $[0, 1]^G$.
\end{theoremi}

To the best of our knowledge, the theorem above is the first shift embeddability result for nonamenable groups. We note that topologically amenable actions cannot have any invariant probability measures, hence the sofic mean dimension satisfies $\mdim_\Sigma(\GonX) \le 0$ for any sofic approximation $\Sigma$ (\cite{Li13}).

With the importance of URPC firmly established, we now turn to characterizing the actions that possess it. The combination of \cite{GutLinTsu16} (where the authors introduced the dynamical tiling technique mentioned in item (i) of the proof of Theorem \ref{thmi: GQT}) and \cite{Niu23} (as well as Lemma \ref{lem: URP implies marker}) gives the following.

\begin{theorem}\label{thmi: GLT}
A free action $\Zb^d \act X$ has URPC if and only if it has the marker property.
\end{theorem}

As in turns out, there is an obstruction to extending this result to actions of general amenable group. For a general action $\GonX$, if a set $U \subset X$ is $F$-free for some $F \subset G$ (i.~e., all the translates $\{gU\}_{g \in F}$ are pairwise disjoint) and $h \in G$ is some other element, then the translate $hU$ will typically not be $F$-free, unless $G$ is abelian (or FC). The fact that this holds for $\Zb^d$ is (implicitly) crucial to the proof of Theorem \ref{thmi: GLT}. 

For a more heuristic explanation, the dynamical tiling construction of Gutman, Lindenstrauss, and Tsukamoto is fundamentally a sophisticated version of Rokhlin's first return time argument. To apply this kind of construction one needs a sufficiently small marker set, hence the marker property naturally appears. In contrast, the proof of Ornstein-Weiss' generalization of the Rokhlin lemma starts with a sufficiently fine \emph{partition}. As we attempted to work with general amenable groups, eventually it became clear that the assumption of the marker property has to be replaced with something stronger.

To this end we introduce a new condition which we call \emph{property FCSB} as well as its weaker version \emph{property FCSB in measure} (Definition \ref{def: FCSB}). Roughly speaking, it asks for an existence of a cover (see Remark \ref{rem: FCSB can ask cover}) consisting of sufficiently free sets such that the combined boundaries are small together. We show (Propositions \ref{prop: FC marker has FCSB} and \ref{prop: URPC implies FCSB}) that 
\begin{itemize}
    \item if $G$ is abelian (more generally, FC) then the marker property, property FCSB, and property FCSB in measure are equivalent,
    \item URP implies property FCSB in measure, and
    \item URPC implies property FCSB.
\end{itemize}

Our next main result is the converse to the second implication above and extend the work of Gutman, Lindenstrauss, and Tsukamoto.

\begin{theoremi}\label{thmi: FCSB in meas URP}
If a free action $\GonX$ of an amenable group has property FCSB in measure then it has URP.
\end{theoremi}

Next, let $\cS$ be the smallest class of \emph{infinite} amenable groups that satisfies the following conditions.
\begin{itemize}
    \item Any group of locally subexponential growth is in $\cS$.
    \item The class $\cS$ is closed under taking direct limits.
    \item If $H \in \cS$ and $H \lhd G$ then $G \in \cS$.
    \item If $G$ contains finite normal subgroups of arbitrarily large cardinality then $G \in \cS$.
    \item Certain amenable groups of dynamical origin belong to $\cS$ (see \cite{NarPet24} for a precise description).
\end{itemize}
We remark that we do not know of any amenable group with a relatively explicit description that is not known to belong to $\cS$. For groups in this class we can obtain the following.

\begin{theoremi}\label{thmi: FCSB URPC}
Let $G \in \cS$. If a free action $\GonX$ has property FCSB then it has URPC.
\end{theoremi}

In particular, we get the following corollary as a special case. Combined with Theorem \ref{thmi: shift embed} this recovers the shift embeddability result of Gutman, Qiao, and Tsukamoto.

\begin{corollaryi}
Let $\GonX$ be a free action of an abelian (or FC) group. Then it has the marker property if and only if it has URPC. 
\end{corollaryi}

As mentioned earlier, the proof of Theorem \ref{thmi: GLT} is based on the first return time argument which fails for actions of more general groups. Consequently, we take a different approach to obtain Theorems \ref{thmi: FCSB in meas URP} and \ref{thmi: FCSB URPC}. Specifically, the proof consists of two parts. In the first one, we show that the number of open sets in the collection given by property FCSB (in measure) can be made to be \emph{independent} of the smallness of the boundaries. This is done using the \emph{cover reduction} procedures that decrease the number of elements in the collection while only increasing the error in a controlled way --- first for actions of cyclic groups (Lemma \ref{lem: alphabet reduction}) and then for actions of arbitrary ones (Theorem \ref{thm: alphabet reduction}). Once this is done, we apply the standard Ornstein-Weiss algorithm to construct a castle with F{\o}lner shapes (Theorem \ref{thm: FCSB in measure equiv URP}). The number of steps in the algorithm depends only on the number of open sets in the collection and hence, by the newly acquired independence, we may choose the boundaries tiny enough so that the accumulated error after running the algorithm is still small. This gives us Theorem \ref{thmi: FCSB in meas URP}.

The argument above is insufficient to prove Theorem \ref{thmi: FCSB URPC}. The problem does not come from the boundary error (which can be controlled in the same way) but from the fact that the classical Ornstein-Weiss algorithm produces a castle that is a priori only large \emph{in measure}. To bridge this gap we recall the notion of a \emph{local algorithm} (Definition \ref{def: local map}). This concept has long been used in measurable combinatorics, as it is not difficult to see that whenever one can solve a discrete problem using a local algorithm, one can also solve a measurable version of the problem (for example, see \cite{Ber23}). As it turns out, this notion is applicable in our context we well. Specifically, we define what it means for a group to have an \emph{exact local tiling algorithm} (Definition \ref{def: local tiling alg}) and show that for such groups property FCSB implies URPC (Theorem \ref{thm: FCSB and local alg imply URPC}). Exact tilings were extensively studied before \cite{KerSza20}, \cite{KerNar21}, \cite{Nar24}, \cite{NarPet24} (although not in the language of local algorithms) and it was (implicitly) shown that the groups from the class $\cS$ defined above have this property. This finishes the proof of Theorem \ref{thmi: FCSB URPC}.

The results above naturally bring up an important question --- if $G$ is not FC, is there indeed a genuine gap between the marker property and FCSB? We suspect that the answer to this question is positive.

\begin{problemi}
Let $G$ be an amenable group with an infinite conjugacy class. Let $\GonX$ be a free action with the marker property. Does it necessarily have property FCSB (in measure)? What if $\GonX$ is additionally assumed to be minimal?
\end{problemi}

The paper is organized as follows. In Section \ref{sec: prelim} we introduce the notation and recall some basic facts about Banach densities. In Section \ref{sec: marker prop and URPC} we develop the machinery for calculations in the type semigroup using the marker sets, observe some basic facts about URP and URPC, and give several equivalent formulations of the latter. In Section \ref{sec: FCSB} we introduce property FCSB (in measure), show how it relates to the marker property, and describe the cover reduction procedure. In Section \ref{sec: from FCSB to URPC} we develop the theory of local tiling algorithms and prove Theorems \ref{thmi: FCSB in meas URP} and \ref{thmi: FCSB URPC}. In Section \ref{sec: mdim and shift embed} we show how a sequence of URPC witnesses can be encoded on a small marker set and prove Theorem \ref{thmi: shift embed}. Finally, in Section \ref{sec: nonamen shift embed} we recall the results about topologically amenable actions of nonamenable groups and prove Theorem \ref{thmi: shift embed nonamen}.

\medskip

\noindent{\it Acknowledgements.}
The research was partially funded by the Deutsche Forschungsgemeinschaft (DFG, German Research Foundation) under Germany’s Excellence Strategy – EXC 2044 – 390685587, Mathematics Münster – Dynamics – Geometry – Structure; the Deutsche Forschungsgemeinschaft (DFG, German Research Foundation) – Project-ID 427320536 – SFB 1442, and ERC Advanced Grant 834267 - AMAREC. It was also partially funded by Dynasnet European Research Council Synergy project -- grant number ERC-2018-SYG 810115.

\section{Preliminaries}
\label{sec: prelim}

Throughout this paper, $X$ will be a compact metrizable topological space, $G$ will be an \emph{infinite} countable discrete group, and $G \act X$ will be a free (that is, all stabilizers are trivial) action by homeomoprhisms. Everywhere except Section \ref{sec: nonamen shift embed} the group $G$ is assumed to be amenable. We sometimes write $F \finsub G$ to say that $F$ is a finite subset of $G$. We will denote by $M_G(X)$ the space of $G$-invariant probability measure on $X$. When $G$ is amenable, $M_G(X)$ is nonempty.

For convenience, we fix a compatible metric $\rho$ on $X$ (note that it is not assumed to be invariant under the group action). Given a set $A \subset X$ and $\delta > 0$, we define 
\[
A^\delta = \{x \in X \mid \rho(x, A) < \delta\} \quad \mbox{and} \quad A^{-\delta} = \{x \in X \mid \rho(x, X \setminus A) > \delta\}.
\]
Note that $A^\delta$ and $A^{-\delta}$ are open for any set $A \subset X$.

Let $F \subset G$ be finite and $\eps >0$. We say that a finite set $S \subset G$ is $(F, \eps)$-invariant if 
\[
\abs{\bigcap_{g \in F}gS} > (1-\eps) \abs{S}.
\]
Recall that a group is amenable if and only if there is an $(F, \eps)$-invariant set for every $F$ and $\eps$. 

\begin{definition}\label{def: Banach density}
Let $Y$ be an arbitrary set (for our purposes, either $Y = G$ or $Y=X$) and let $G \act Y$ be a free action. Given a set $A \subset Y$ and a nonempty finite set $F \subset G$ define
\[
\lD_F(A) = \inf_{y \in Y}\frac{\abs{A \cap Fy}}{\abs{F}} \quad \mbox{and} \quad \uD_F(A) = \sup_{y \in Y}\frac{\abs{A \cap Fy}}{\abs{F}}.
\]
We refer to these quantities as \emph{lower and upper Banach densities of $A$ with respect to $F$}. We define \emph{lower and upper Banach densities of $A$} by
\[
\lD(A) = \sup_{F \finsub G}\lD_F(A) \quad \mbox{and} \quad \uD(A) = \inf_{F \finsub G} \uD_F(A).
\]
\end{definition}
If $(F_n)$ is a F{\o}lner sequence for $G$ then
\[
\lD(A) = \limsup_{n \to \infty}\lD_{F_n}(A) \quad \mbox{and} \quad \uD(A) = \liminf_{n \to \infty} \uD_{F_n}(A).
\]
Note that the upper Banach density is also known as the \emph{orbit capacity}. Suppose now that $Y = X$ is a compact space and $G$ acts on it by homeomorphisms. If $A \subset X$ is open we have 
\[
\lD(A) = \inf_{\mu \in M_G(X)} \mu(A),
\]
and if $A \subset X$ is closed we have
\[
\uD(A) = \sup_{\mu \in M_G(X)} \mu (A).
\]

For arbitrary subsets $S \subset G$ and $V \subset X$ we define
\[
SV = \bigcup_{g \in S}gV.
\]

If $\{\cU_i\}_{i=1}^n$ are collections of open sets, we define the join
\[
\bigvee_{i=1}^n \cU_i = \left\{U_1 \cap U_2 \cap \ldots \cap U_n \mid U_i \in \cU_i\right\}.
\]
If $\cU$ is a collection of open sets in $X$ and $Y \subset X$, we define the restriction
\[
\left.\cU\right|_{Y} = \{U \cap Y \mid U \in \cU\}.
\]
If $g \in G$, we define the collection $g\cU$ to be
\[
g\cU = \{gU \mid u \in \cU\}.
\]

\section{Marker property, Uniform Rokhlin Property, and comparison}
\label{sec: marker prop and URPC}

\subsection{Dynamical subequivalence, type semigroup, and marker sets}
\begin{definition}\label{def: L-free}
Given a set $U \subset X$ and a finite set $L \subset G$ we say that $U$ is \emph{$L$-free} if all the translates $\{gU\}_{g \in L}$ are pairwise disjoint.
\end{definition}

\begin{definition}\label{def: marker prop}
We say that an open set $U \subset X$ is a \emph{marker set} if $GU = X$. The action $G \act X$ has the \emph{marker property} if for every finite subset $L \subset G$ there exists an $L$-free marker set.
\end{definition}

\begin{remark}
By compactness, if $U$ is a marker set then there exists some finite subset $M \subset G$ with $MU = X$. In particular, that implies that 
\[
\lD(U) \ge 1/\abs{M}.
\]

Conversely, for any open set $U$, the set $X \setminus GU$ is a $G$-invariant compact set and thus carries a $G$-invariant probability measure. It follows that if $\lD(U) > 0$, then $U$ is a marker set.
\end{remark}

The following lemma shows that whenever we choose an $L$-free marker set $U$ we can in fact assume that $\overline{U}$ is $L$-free. The proof is elementary and is left as an exercise.

\begin{lemma}
\label{lem: MS shrink}
Let $U$ be an $L$-free marker set such that $MU = X$ for some $M \finsub G$. Then there exists $\delta > 0$ such that $U^{-\delta}$ is an $L$-free marker set with $MU^{-\delta} = X$. In particular, $\overline{U^{-\delta}} \subset U$ is $L$-free.
\end{lemma}

\begin{definition}
Let $A \subset X$ be closed and let $B \subset X$ be open. We say that $A$ is \emph{(dynamically) subequivalent} to $B$ and write $A \prec B$ if there exist a finite open cover $A \subset \bigcup_{j \in J}U_j$ and elements $g_j \in G, j \in J$ such that $\{g_jU_j\}_{j\in J}$ are pairwise disjoint subsets of $B$.

For an open set $A$ we say that it is subequivalent to $B$ (and write again $A \prec B$) if for every closed subset $A^\prime \subset A$ we have $A^\prime \prec B$. Note that if $A$ is open and $A \subset B$ then $A \prec B$. Moreover, we have that $gA \prec A$ and $A \prec gA$ for any $g \in G$.

It is an easy exercise to check that the dynamical subequivalence is transitive and therefore defines a partial preorder on the set of open subsets of $X$.

Note that if $A \prec B$ then $\mu(A) \le \mu(B)$ for all $\mu \in M_G(X)$.
\end{definition}

\begin{remark}\label{rem: subeq cover properties}
Let $A$ be a closed set and let $B$ be an open set. Suppose that $A \prec B$ and this subequivalence is implemented by the cover $\{U_j\}_{j \in J}$ together with elements $g_j \in G$. Then for any given $\delta > 0$ we can always assume that
\begin{enumerate}
    \item $g_j\overline{U_j}$ are pairwise disjoint in $B$ and
    \item $\bigcup_{j \in J}U_j \subset A^\delta$.
\end{enumerate}
Indeed, item (i) is achieved by slightly shrinking sets $U_j$ and item (ii) is achieved by intersecting them with $A^\delta$.
\end{remark}

\begin{definition}\label{def: type semigroup}
The \emph{type semigroup} of an action $G \act X$ is the abelian semigroup generated by all open subsets of $X$ subject to the relations 
\[
[U \sqcup V] = [U] + [V] \quad \mbox{and} \quad [gU] = [U]
\]
for all $g \in G$. The dynamical subequivalence respects these relations and therefore extends to a partial preorder on the type semigroup (which we once again denote by $\prec$) that is compatible with the semigroup structure. We will frequently not make a distinction between a set $U$ and the element $[U]$.
\end{definition}

We remark that our definition is slightly different from the definition of the \emph{generalized type semigroup} introduced by Ma \cite{Ma21}. More specifically, we do not take a further quotient to obtain a semigroup with a genuine partial order on it. However, in the present paper we are not interested in the properties of the type semigroup itself and only use it to write calculations in a more convenient manner. Thus, the notion above is sufficient for our purposes.

\begin{definition}\label{def: closed set in type semigroup}
Let $\{A_i\}_{i=1}^n$ be closed sets. Employing a slight abuse of notation, we will write subequivalence expressions involving the terms $[A_i]$ if they hold with $A_i$ replaced by some sufficiently small neighbourhoods $A_i^{\delta_i}$.

As an example, if $A_1, A_2$ are closed and $B_1, B_2$ are open then the expression $[A_1] + [B_1] \prec [A_2] + [B_2]$ means that for any $\delta > 0$ there are $0 < \delta_1, \delta_2 < \delta$ such that $[A_1^{\delta_1}] + [B_1] \prec [A_2^{\delta_2}] + [B_2]$. 

We write an expression involving several consecutive subequivalences (such as $[A_1] \prec [A_2] \prec [A_3]$) if each of them holds in the above sense. One may check that the relation $\prec$ is still transitive and thus the calculations can be safely carried out.

\end{definition}

It turns out that there is a nice connection between dynamical subequivalence and the marker property. Specifically, the marker sets can be viewed as arbitrarily small everywhere positive elements in the type semigroup. This idea is formalized in the next few lemmas.

\begin{lemma}
\label{lem: L-free subeq}
Suppose $A$ is an $L$-free open subset of $X$ and $B$ is an open subset such that $L^{-1}B = X$. Then $A \prec B$.
\end{lemma}

\begin{proof}
For $g \in L$ set $U_g = \brak{g^{-1}B} \cap A$. Then 
\[
\bigcup_{g \in L}U_g = \bigcup_{g \in L}\brak{g^{-1}B} \cap A = A,
\]
and $gU_g = B \cap gA$ are pairwise disjoint subsets of $B$.
\end{proof}

\begin{remark}\label{rem: image subeq}
Note that in the previous lemma, by setting 
\[
V = \bigsqcup_{g \in L}gU_g,
\]
we can in fact obtain an open set $V \subset B$ such that $A \prec V \prec \abs{L}[A]$.
\end{remark}

The following easy statement is the cornerstone of our techniques.

\begin{lemma}\label{Lem: mult L-free subeq}
Let $G$ be a countably infinite group, let $k \in \Nb$ and let $B$ be a marker set. Then there exists some finite $L \subset G$ such that for any $L$-free open set $U$
\[
k[U] \prec [B].
\]
\end{lemma}

\begin{proof}
Since $B$ is a marker set there exists some finite $M \subset G$ containing the identity such that $MB = X$. Choose $L$ so that there are elements $g_1, g_2, \ldots, g_k \in G$ such that $M^{-1}g_i$ are pairwise disjoint subsets of $L$. Then for any $L$-free set $U$
\[
k[U] = \sum_{i=1}^k [g_iU]= \sqbrak{\bigsqcup_{i=1}^k g_iU} \prec B
\]
by Lemma \ref{lem: L-free subeq}.
\end{proof}

Some techniques for establishing dynamical subequivalence were developed in \cite{Nar22}. We recall the central lemma \cite[Lemma 3.1]{Nar22} of that work.

\begin{lemma}\label{lem: polygrowth subeq}
Let $A \subset X$ be closed and $B \subset X$ open. Assume that for some $\eps >0$, finite $D \subset G$, and $k \in \Nb$ we have
\[
\abs{\{g \in D^{-1}D \mid gx \in A^{\eps}\}} < k\abs{\{g \in D \mid gx \in B^{-\eps}\}}
\]
for all $x \in X$. Then $[A] \prec k[B]$.
\end{lemma}

The result above straightforwardly extends to show that that
\[
\sum_{i=1}^n [A_i] \prec k\sum_{j=1}^m [B_j]
\]
for a collection of closed sets $\{A_i\}_{i=1}^n$ and a collection of open sets $\{B_j\}_{j=1}^m$ as long as we have 
\[
\sum_{i=1}^n\abs{\{g \in D^{-1}D \mid gx \in A_i^{\eps}\}} < k\sum_{j=1}^m\abs{\{g \in D \mid gx \in B_j^{-\eps}\}}
\]
for some $\eps >0$, finite $D \subset G$, and $k \in \Nb$ and all $x \in X$.

The special situation of $D = \{e\}$ and $k=1$ gives a surprisingly useful corollary which we record below. Note that in this case the group action is completely irrelevant.

\begin{corollary}\label{cor: subeq of elements}
Let $\{A_i\}_{i=1}^n$ be a collection of closed sets and let $\{B_j\}_{j=1}^m$ be a collection of open sets. Suppose that for some $\eps > 0$ we have
\[
\abs{\{i \mid x \in A_i^{\eps}\}} < \abs{\{j \mid x \in B_j^{-\eps}\}}
\]
for all $x \in X$. Then $\sum_{i=1}^n [A_i] \prec \sum_{j=1}^m [B_i]$.
\end{corollary}

We end this subsection by recalling the definition of comparison and some of its variations. Here, as usual, $A$ and $A_i$ stand for closed sets and $B$ and $B_j$ stand for open sets.

\begin{definition}\label{def: comparison}
An action $G \act X$ is said to have:
\begin{itemize}
    \item \emph{comparison} if $A \prec B$ whenever $\mu(A) < \mu(B)$ for every $\mu \in M_G(X)$,
    \item \emph{$k$-comparison} (for some $k \in \Nb$) if $[A] \prec (k+1)[B]$ whenever $\mu(A) < \mu(B)$ for every $\mu \in M_G(X)$,
    \item \emph{weak $k$-comparison} (for some $k \in \Nb$) if $[A] \prec (k+1)[B]$ whenever $\uD(A) < \lD(B)$,
    \item \emph{comparison on multisets} if $\sum_{i=1}^{n}[A_i] \prec \sum_{j=1}^{m}[B_j]$ whenever $\sum_{i=1}^{n}\mu(A_i) < \sum_{j=1}^{m}\mu(B_j)$ for every $\mu \in M_G(X)$.
\end{itemize}
\end{definition}

Lemma \ref{lem: polygrowth subeq} was used in \cite{Nar22} to deduce the following main result \cite[Theorem 3.2]{Nar22}.
\begin{theorem}\label{thm: polygrowth comp}
Let $G$ be a group of polynomial growth with $d$ being the order of the growth function. Then $G \act X$ has weak $2^d$-comparison.
\end{theorem}

\subsection{URP and URPC}
In this subsection we introduce the two central properties used in this work. Both of them can be regarded as a topological version of the conclusion of the Ornstein-Weiss tiling theorem. We begin with some basic terminology.

\begin{definition}
A pair $(S, V)$, where $S \subset G$ is finite and $V \subset X$ is open is called a \emph{tower} if the sets $\{gV\}_{g \in S}$ are pairwise disjoint (that is, $V$ is $S$-free). In such a situation, we say that
\begin{itemize}
    \item $V$ is the \emph{base} of the tower.
    \item $S$ is the \emph{shape} of the tower,
    \item $gV$ for $g \in S$ are \emph{levels} of the tower, and
    \item $SV$ is the \emph{footprint} of the tower.
\end{itemize}
A finite collection $\{(S_i, V_i)\}_{i=1}^n$ of towers is called a \emph{castle} if moreover all the footprints $S_iV_i$ are pairwise disjoint. The \emph{footprint} of a castle $\{(S_i, V_i)\}_{i=1}^n$ is the set $\bigsqcup_{i=1}^nS_iV_i$.
\end{definition}

\begin{definition}\label{def: URP}
An action $\GonX$ has the \emph{Uniform Rokhlin Property (URP)} if for every finite set $K \subset G$ and every $\eps > 0$ there is a castle $\{(S_i, V_i)\}_{i=1}^n$ such that the shapes $S_i$ are $(K, \eps)$-invariant and
\[
     \uD\brak{X \setminus\bigsqcup_{i=1}^nS_iV_i} < \eps.
\]
\end{definition}

We remark that the only difference between URP and \emph{almost finiteness in measure} (see \cite{KerSza20}) is that the latter additionally asks for all levels $gV_i, g \in S_i$ of the castle to be small in diameter. This seemingly small distinction turns out to be quite significant. Indeed, Kerr ans Szab{\'o} showed that for free actions almost finiteness in measure is equivalent to the small boundary property (SBP). On the other hand, there are many actions that have URP but not SBP.

\begin{lemma}\label{lem: URP implies marker}
Suppose an action $\GonX$ has URP. Then it has the marker property.
\end{lemma}

\begin{proof}
Let $L$ be a finite subset of $G$ containing the identity. Let $\{(S_i, V_i)\}_{i=1}^n$ be an open castle with $(L^{-1}, 1/2)$-invariant shapes such that 
\[
\uD\brak{X \setminus \bigsqcup_{i=1}^nS_iV_i} < 1/2.
\]
By the invariance condition, for each $i=1, 2, \ldots, n$ there is an element $g_i \in S_i$ such that $Lg_i \subset S_i$. Set
\[
U = \bigsqcup_{i=1}^ng_iV_i.
\]
Then $U$ is $L$-free and $\lD(U) > 0$ which means it is a marker set.
\end{proof}

\begin{lemma}\label{lem: eps levels subeq marker}
Let $U$ be a marker set. Then there exists an $\eps > 0$ and a finite set $K \subset G$ such that for any castle $\{(S_i, V_i)\}_{i=1}^n$ with $(K, \eps)$-invariant shapes we have
\[
\sum_{i=1}^n\lceil\eps\abs{S_i}\rceil [V_i] \prec U.
\]
\end{lemma}

\begin{proof}
Find a finite set $L \subset G$ containing the identity such that $L^{-1}U = X$. Set $K = L^{-1}L$ and choose $\eps$ so that $(1-\eps)/\abs{K} > \eps$. Let $S_i$ be a $(K, \eps)$-invariant finite subset of $G$. Then 
\[
\abs{\bigcap_{g \in K}gS_i} > (1 - \eps)\abs{S_i}
\]
and we can choose a subset $S_i^\prime \subset \bigcap_{g \in K}gS_i$ which is $K$-independent and satisfies
\[
\abs{S_i^\prime} \ge \frac{\abs{\bigcap_{g \in K}gS_i}}{\abs{K}} > \frac{(1-\eps)}{\abs{K}}\abs{S_i} > \eps \abs{S_i}.
\]
It follows that
\[
\sum_{i \in I}\lceil\eps\abs{S_i}\rceil[V_i] \prec \sum_{i \in I}S_i^\prime [V_i].
\]
However, $\bigsqcup_{i \in I}S_i^\prime V_i$ is an $L$-free open set and therefore subequivalent to $U$ by Lemma \ref{lem: L-free subeq}.
\end{proof}

We are now ready to prove the main theorem of this section, which gives several equivalent definitions for the conjunction of URP and comparison.

\begin{theorem}
\label{thm: URP+comp equiv}
For an action $G \act X$ the following are equivalent.
\begin{enumerate}
    \item It has URP and comparison.
    \item It has URP and weak $k$-comparison for some $k \in \Nb$.
    \item For every finite $K \subset G$ and $\eps > 0$ there is a castle $\{(S_i, V_i)\}_{i=1}^n$ with $(K, \eps)$-invariant shapes such that
    \[
    X \setminus \bigsqcup_{i=1}^nS_iV_i \prec \sum_{i=1}^n\lfloor \eps\abs{S_i}\rfloor [V_i].
    \]
    \item It has the marker property and for every finite $K \subset G$, $\eps > 0$, and a marker set $U$ there is an open castle $\{(S_i, V_i)\}_{i=1}^n$ with $(K, \eps)$-invariant shapes such that
    \[
    X \setminus \bigsqcup_{i=1}^nS_iV_i \prec U.
    \]
\end{enumerate}
\end{theorem}

\begin{proof}
It is clear that (i) implies (ii). 

To see that (ii) implies (iii), use URP to choose a castle $\{(S_i, V_i)\}_{i=1}^n$ such that the shapes $S_i$ are $(K, \eps)$-invariant with $\abs{S_i} > (k+1)/\eps$ and
\[
\uD\brak{X \setminus \bigsqcup_{i=1}^nS_iV_i} < \frac{\eps(1-\eps)}{2(k+1)} < \eps.
\]
For each $i=1, 2, \ldots, n$ choose a subset $S_i^\prime \subset S_i$ with
\[
\abs{S_i^\prime} = \left\lfloor \frac{\eps\abs{S_i}}{k+1}\right\rfloor > \frac{\eps\abs{S_i}}{2(k+1)}.
\]
It follows that 
\[
\lD\brak{\bigsqcup_{i=1}^nS_i^\prime V_i} > \frac{\eps(1-\eps)}{2(k+1)} > \uD\brak{X \setminus \bigsqcup_{i=1}^nS_iV_i}
\]
and thus, using weak $k$-comparison,
\[
X \setminus \bigsqcup_{i=1}^nS_iV_i \prec (k+1)\left[\bigsqcup_{i=1}^nS_i^\prime V_i\right] \prec \sum_{i=1}^n\lfloor \eps\abs{S_i}\rfloor [V_i].
\]

Lemmas \ref{lem: URP implies marker} and \ref{lem: eps levels subeq marker} show that (iii) implies (iv). 

It remains to show that (iv) implies (i). Clearly, URP holds and therefore we only need to show comparison. Let $A$ be a closed set and $B$ an open set such that $\mu(A) < \mu(B)$ for every $G$-invariant probability measure $\mu$. By \cite[Lemma 3.3]{Ker20} there exists some $\delta > 0$ such that in fact 
\[
\mu(A^\delta) + \delta < \mu(B^{-\delta})
\]
for every $G$-invariant probability measure $\mu$. Note that, in particular, $\mu(B) > \delta$ for every $\mu \in M_G(X)$ and thus $\lD(B) > 0$ which means that it is a marker set. It follows that there is some finite set $L \subset G$ such that $L^{-1}B = X$. Find a marker set $U$ with $\uD(\overline{U}) < \delta/\abs{L}$. By slightly shrinking the set $U$ (Lemma \ref{lem: MS shrink}) and using Remark \ref{rem: image subeq} we can obtain an open set $V \subset B$ such that $V$ is a marker set and $\uD(\overline{V}) < \delta$. Set $B^\prime = B \setminus \overline{V}$. Then 
\[
B^\prime \sqcup V \subset B
\]
and
\[
\mu(A) < \mu(B^\prime)
\]
for every $\mu \in M_G(X)$. The last inequality implies that there exist some finite $K \subset G$ and $\eps > 0$ such that for every $(K, \eps)$-invariant set $S \subset G$ and every $x \in X$ we have
\begin{equation}\label{eq: counting ineq}
\sum_{g \in S}\delta_x(g^{-1}A) < \sum_{g \in S}\delta_x(g^{-1}B^\prime),    
\end{equation}
where $\delta_x$ is the Dirac measure at $x$. Choose an open castle $\{(S_i, V_i)\}_{i=1}^n$ with $(K, \eps)$-invariant shapes containing the identity such that
\[
X \setminus \bigsqcup_{i=1}^nS_iV_i \prec V.
\]
Moreover, find a small enough $\tau > 0$ such that
\[
X \setminus \bigsqcup_{i=1}^nS_iV_i^{-\tau} \prec V.
\]
Set $Y = \bigsqcup \overline{V_i^{-\tau/2}}$. Then, by using \eqref{eq: counting ineq} and applying Corollary \ref{cor: subeq of elements} on the set $Y$, we have
\begin{multline*}
A \prec \sum_{i=1}^n \sum_{g \in S_i} \left[g\overline{V_i^{-2\tau/3}} \cap A\right] + X \setminus \bigsqcup_{i=1}^nS_iV_i^{-\tau} \\ \prec \sum_{i=1}^n \sum_{g \in S_i} \left[\overline{V_i^{-2\tau/3}} \cap g^{-1}A\right] + [V] \\ \prec \sum_{i=1}^n \sum_{g \in S_i} \left[V_i^{-\tau/2} \cap g^{-1}B^\prime\right] + [V] \prec [B^\prime] + [V] \prec B.
\end{multline*}

\end{proof}

\begin{definition}
We say that an action $\GonX$ has \emph{URPC} (Uniform Rokhlin Property and comparison) if either of the equivalent properties in Theorem \ref{thm: URP+comp equiv} hold.
\end{definition}

\begin{remark}
Similarly to how URP is a direct weakening of almost finiteness in measure, URPC is a direct weakening of a property known as \emph{almost finiteness} (see \cite{Ker20}). The latter holds automatically for many amenable groups acting freely on finite-dimensional spaces (see, for instance, \cite{Nar24}). In particular, the corollary below is widely applicable.
\end{remark}

\begin{corollary}\label{cor: URPC passes to ext}
Suppose that an action has URPC. Then any extension of this action also has URPC.
\end{corollary}

\begin{proof}
Taking preimage under the quotient map shows that item (iii) in Theorem \ref{thm: URP+comp equiv} passes to extensions.
\end{proof}

We finish the section by introducing several auxiliary definitions that will be useful later.

\begin{definition}\label{def: URPC witness}
Suppose that the action $\GonX$ has URPC and let $K \subset G$ be finite and $\eps > 0$. We say that the data $\cW = \brak{\{S_i\}_{i=1}^n, \{V_i\}_{i=1}^n, \{g_j\}_{j=1}^m, \{U_j\}_{j=1}^m}$ is a \emph{$(K, \eps)$-URPC witness} if
\begin{itemize}
    \item $\{(S_i, V_i)\}_{i=1}^n$ is a castle and satisfies item (iii) in Theorem \ref{thm: URP+comp equiv},
    \item $\{U_j\}_{j=1}^m$ is an open cover of $X \setminus \bigsqcup_{i=1}^nS_iV_i$ and $g_j$ are elements of G,
    \item there is a partition of the set $\{1, 2, \ldots, m\}$ into (possibly empty) subsets $J_{i, k}$, where $i \in \{1, 2, \ldots, n\}$ and $k \in \{1, 2, \ldots, \lfloor \eps\abs{S_i}\rfloor\}$ such that $\{g_jU_j\}_{j \in J_{i, k}}$ are pairwise disjoint subsets of $V_i$.
\end{itemize}
In other words, the data $\{(g_j, U_j)\}_{j =1}^m$ witnesses the subequivalence
\[
    X \setminus \bigsqcup_{i=1}^nS_iV_i \prec \sum_{i=1}^n\lfloor \eps\abs{S_i}\rfloor [V_i].
\]
Remark \ref{rem: subeq cover properties} shows that we may always assume that in fact the sets $\{g_j\overline{U_j}\}_{j \in J_{i, k}}$ are pairwise disjoint subsets of $V_i$.

Given such a URPC witness, we define the corresponding \emph{URPC cover} to be the collection 
\[
C(\cW) = \{gV_i \colon 1 \le i \le n, g \in S_i\} \sqcup \{U_j\}_{j=1}^m.
\]
Note that it is indeed a cover of $X$. Furthermore, we define the \emph{encoding element} of this cover in the type semigroup to be 
\[
E(\cW) = \sum_{i=1}^n[V_i] + \sum_{j=1}^m[U_j].
\]
\end{definition}

\begin{remark}\label{rem: URPC cover stable}
Similar to Lemma \ref{lem: MS shrink}, URPC witnesses are stable under small shrinking of sets. More precisely, let $\cW = \brak{\{S_i\}_{i=1}^n, \{V_i\}_{i=1}^n, \{g_j\}_{j=1}^m, \{U_j\}_{j=1}^m}$ be a $(K, \eps)$-URPC witness. Then there exists some $\delta > 0$ such that for any opens sets $V_i^{-\delta} \subset V_i^\prime \subset V_i$ and any open sets $U_j^{-\delta} \subset U_j^\prime \subset U_j$ the data $\cW^\prime = \brak{\{S_i\}_{i=1}^n, \{V_i^\prime\}_{i=1}^n, \{g_j\}_{j=1}^m, \{U_j^\prime\}_{j=1}^m}$ is also a $(K, \eps)$-URPC witness.
\end{remark}

\section{Property FCSB}
\label{sec: FCSB}

In this section we introduce and study properties FCSB and FCSB in measure. We remark that most of the statements have two versions: in density, corresponding to URP, and group-combinatorial, corresponding to URPC. We will always give a proof for the second (more difficult) version. The density statements can be obtained by using \cite[Lemma 3.3]{Ker20} and substituting the condition ``$\cdot \prec V$'' with ``$\uD(\cdot) < \eps$''.

\subsection{The definition and the relation to the marker property}

\begin{definition}\label{def: FCSB}
An action $G \act X$ has \emph{property FCSB (free covers with staggered boundaries)} if it has the marker property and for every finite $F \subset G$ and any marker set $V$ there is a collection of $F$-free open sets $\{U_i\}_{i=1}^n$ with
\[
\sum_{i=1}^n \sqbrak{\partial U_i} + \left[X \setminus \bigcup_{i=1}^nU_i\right] \prec V.
\]
Note that here we use the convention introduced in Definition \ref{def: closed set in type semigroup}.

An action $G \act X$ has \emph{property FCSB in measure} if for every finite $F \subset G$ and any $\eps > 0$ there is a collection of $F$-free open sets $\{U_i\}_{i=1}^n$ such that
\[
\sum_{i=1}^n \mu\brak{\partial U_i} + \mu\brak{X \setminus \bigcup_{i=1}^nU_i} <\eps
\]
for any $\mu \in M_G(X)$. It is clear that property FCSB in measure is a direct weakening of property FCSB.
\end{definition}

Note that in the definition above, we do not require that the sets $U_i$ cover $X$ (and instead ask that the remainder $X \setminus \bigcup_{i=1}^nU_i$ is small in an appropriate sense). This is purely a matter of choice, as the following remark shows. The only reason we do not impose this condition is that it will be slightly more convenient for us to allow for a remainder when proving Theorem \ref{thm: alphabet reduction}.

\begin{remark}\label{rem: FCSB can ask cover}
In the definition of property FCSB (in measure) one can equivalently ask for the collection $\{U_i\}_{i=1}^n$ to be a cover of $X$.
\end{remark}

\begin{proof}
We will prove the statement for property FCSB, with the proof for property FCSB in measure being essentially the same. 

Let $F \subset G$ be finite and let $V$ be a marker set. Choose an arbitrary $F$-free cover $\{V_j\}_{j=1}^m$ of X and use Lemma \ref{Lem: mult L-free subeq} to choose a marker set $U$ such that
\[
m[U] \prec [V].
\]
By property FCSB, there is a collection $\{U_i\}_{i=1}^n$ and some $\delta > 0$ such that
\[
\sum_{i=1}^n \sqbrak{\partial U_i} + \left[\brak{X \setminus \bigcup_{i=1}^nU_i}^\delta\right] \prec U.
\]
Consider the collection of open sets 
\[
\{U_i\}_{i=1}^n \sqcup \left\{V_j \cap \brak{X \setminus \bigcup_{i=1}^nU_i}^{\delta/2}\right\}_{j=1}^m.
\]
Clearly, it is a cover of $X$ consisting of $F$-free sets and
\[
\sum_{i=1}^n \sqbrak{\partial U_i} + \sum_{j=1}^m \sqbrak{\partial \brak{V_j \cap \brak{X \setminus \bigcup_{i=1}^nU_i}^{\delta/2}}} \prec \sum_{i=1}^n \sqbrak{\partial U_i} + m\left[\brak{X \setminus \bigcup_{i=1}^nU_i}^\delta\right] \prec m[U] \prec V.
\]
\end{proof}

Recall that a group $G$ is said to be \emph{FC} if every conjugacy class of $G$ is finite. Any abelian group is FC.

\begin{proposition}\label{prop: FC marker has FCSB}
If $G$ is FC and an action $\GonX$ has the marker property then it has property FCSB. 
\end{proposition}

\begin{proof}
Let $F \subset G$ be finite and let $V$ be a marker set. There exists some finite set $M \subset G$ such that that $M^{-1}V = X$. Since $G$ is FC the set 
\[
L = \bigcup_{g \in G}g(F \cup M)g^{-1}
\]
is finite and conjugation invariant. Let $U$ be an $L$-free marker set. By Lemma \ref{lem: MS shrink} there is $\delta > 0$ such that $U^{-\delta}$ is also an $L$-free marker set. Choose a bump function $\phi : X \to [0, 1]$ with $\supp\phi \subset U$ and $\phi = 1$ on $U^{-\delta}$. Enumerate all elements of $G = \{g_1, g_2, \ldots\}$ and define 
\[
U_i = g_i\{x \in X \mid \phi(x) > 1/i\} = \{x \in X \mid \phi(g_i^{-1}x) > 1/i\}.
\]
By conjugation invariance of $L$, the sets $U_i$ are $L$-free and cover $X$. By compactness, there is some $n \in \Nb$ such that $\{U_i\}_{i=1}^n$ is also a cover of $X$. Moreover, $g_i^{-1}\brak{\partial U_i} \subset \{x \in X \mid \phi(x) = 1/i\}$ are disjoint subsets of $U$. Thus, using Lemma \ref{lem: L-free subeq},
\[
\sum_{i=1}^n [\partial U_i] \prec U \prec V.
\]
\end{proof}

We remark that even if the group $G$ is not abelian, if an element $g \in G$ lies in the center $Z(G)$, we still have the equality
\[
gFg^{-1} = F.
\]
With this in mind, a slight modification of the proof above shows the following.

\begin{proposition}
Let $Z(G)$ be the center of the group $G$ and suppose that the restricted action $Z(G) \act X$ is minimal. Then $G \act X$ has property FCSB.
\end{proposition}

\begin{proposition}\label{prop: URPC implies FCSB}
If an action $G \act X$ has URPC then it has property FCSB. If it has URP then it has property FCSB in measure.
\end{proposition}

\begin{proof}
We again prove the statement for property FCSB and URPC, with the proof for property FCSB in measure and URP being similar. 

By Lemma \ref{lem: URP implies marker}, if an action has URPC then it has the marker property. Let $F \subset G$ be finite and let $V$ be a marker set. Choose a marker set $U$ such that
\[
3[U] \prec V.
\]
Using URPC and Lemma \ref{lem: eps levels subeq marker}, find an $\eps > 0$ and a castle $\{(S_i, V_i)\}_{i=1}^n$ with $(F, \eps)$-invariant shapes such that 
\[
\sum_{i=1}^n\lceil\eps\abs{S_i}\rceil [V_i] \prec U \quad \mbox{and} \quad X \setminus \bigsqcup_{i=1}^nS_iV_i \prec U.
\]
By Remark \ref{rem: URPC cover stable} we can additionally assume that the closures $g\overline{V_i}, g \in S_i$ of all the levels of the castle are pairwise disjoint and therefore their boundaries $\partial \brak{gV_i}$ are pairwise disjoint subsets of the remainder $X \setminus \bigsqcup_{i=1}^nS_iV_i$. For each shape $S_i$ define
\[
S_i^\prime = \{g \in S_i \mid Fg \subset S_i\}
\]
and note that $\abs{S_i^\prime} > (1-\eps)\abs{S_i}$ by $(F, \eps)$-invariance. Consider the following collection of open sets
\[
\{gV_i \mid 1 \le i \le n, g \in S_i^\prime\}.
\]
Then each of them is $F$-free and 
\[
\sum_{i=1}^n \sum_{g \in S_i^\prime} \sqbrak{\partial \brak{gV_i}} + \left[X \setminus \bigcup_{i=1}^n S_i^\prime V_i\right] \prec \sqbrak{X \setminus \bigsqcup_{i=1}^nS_iV_i} + \sum_{i=1}^n\lceil\eps\abs{S_i}\rceil [V_i] + \sqbrak{X \setminus \bigsqcup_{i=1}^nS_iV_i} \prec 3[U] \prec V.
\]
\end{proof}

\subsection{Cover reduction}

The goal of this section is to show that the number of sets $U_i$ in the Definition \ref{def: FCSB} can be made to only depend on the set $F$. We begin with a lemma that reduces a given cover in a ``single direction'' --- that is, for an action of a cyclic group generated by some element of $G$. We remark that the subequivalence below is with respect to the action of the cyclic group and thus also holds for the original action $\GonX$.

\begin{lemma}\label{lem: alphabet reduction}
Let $\langle g \rangle \act X$ be an action of a cyclic group. Let $\{U_i\}_{i=1}^n$ be a collection of open $\{e, g\}$-free sets and let $\delta > 0$. Then there exist open $\{e, g\}$-free sets $V_1, V_2$, and $V_3$ such that their closures $\overline{V_1}, \overline{V_2},$ and $\overline{V_3}$ are pairwise disjoint and
\[
[\partial V_1] + [\partial V_2] + [\partial V_3] + \sqbrak{X \setminus (V_1 \cup V_2 \cup V_3)} \prec 8\brak{\sum_{i=1}^n \sqbrak{\brak{\partial U_i}^\delta} + \sqbrak{X \setminus \bigcup_{i=1}^nU_i^{-\delta}}}.
\]
\end{lemma}

\begin{proof}
For each $i=1 , \ldots, n$ choose a bump function $\phi_i$ which is supported on $U_i$ and such that $\phi_i \equiv 1$ on $U_i \setminus \brak{\partial U_i}^\delta$. In particular, we have that
\begin{equation}\label{eq: bump func and boundary neighb}
\left\{x \mid 0 < \phi_i(x) < 1\right\} \subset \brak{\partial U_i}^\delta.
\end{equation}
Set
\[
U_i^\prime = \left\{x \mid \phi_i(x) > \frac{1}{2}\right\}
\]
and note that $U_i^{-\delta} \subset U_i^\prime \subset U_i$. For each $m \in \Zb$ set $\cU_i^m$ to be a collection consisting of two open sets
\[
g^m\left\{x \mid \phi_i(x) > \frac{1}{\abs{m}+2}\right\} \ \mbox{and} \ g^m\left\{x \mid \phi_i(x) < \frac{1}{\abs{m}+3}\right\}.
\]
We remark that $U_i^\prime \in \cU_i^0$ and that the closures of the two sets in each $\cU_i^m$ are disjoint. An important property of these collections is that if $U \in \cU_i^{m-1}$ then $gU$ intersects exactly one set from the pair $\cU_i^m$ and $g^{-1}U$ intersects exactly one set from the pair $\cU_i^{m-2}$. Moreover, 
\begin{equation}\label{eq: staggered gaps subeq}
g^{-m}\brak{X \setminus \bigcup_{U \in \cU_i^m} U} = \left\{x \mid \phi_i(x) \in \left[\frac{1}{\abs{m}+3}, \frac{1}{\abs{m}+2}\right]\right\}.
\end{equation}
It follows that for any $M \in \Zb_{\ge 0}$, setting 
\[
\cV_i^M = \bigvee_{m=-M}^{M} \cU_i^{m}
\]
and using \eqref{eq: bump func and boundary neighb} and \eqref{eq: staggered gaps subeq}, we have
\[
X \setminus \bigcup_{U \in \cV_i^M}U = \bigcup_{m=-M}^{M} \brak{X \setminus \bigcup_{U \in \cU_i^m} U} \prec 4\sqbrak{\brak{\partial U_i}^\delta}.
\]
Finally, for $k = 1, \ldots, n$, let
\[
\cV_k = \bigvee_{i=1}^k \cV_{i}^{k-i}
\]
and note that $U_k^\prime \in \cU_k^0$, which allows us to define a subcollection of $\cV_k$ by
\[
\cV_k^\prime = \brak{\bigvee_{i=1}^{k-1} \cV_{i}^{k-i}} \vee \{U_k^\prime\} = \left\{U \in \cV_k \mid U \subset U_k^\prime\right\}.
\]
We observe the following easy facts. The first four of them are clear and the last one follows from the property of the collections $\cU_i^m$ mentioned above.
\begin{enumerate}
    \item For any distinct $U, V \in \cV_k^\prime$ the closures $\overline{U}$ and $\overline{V}$ are disjoint.
    \item Any set $U \in \cV_k^\prime$ is $\{e, g\}$-free.
    \item If $U, V$ are sets in $\cV_k^\prime$ then $gU \cap V = \emptyset$.
    \item If $k^\prime < k$ then $\cV_{k}$ refines $\cV_{k^\prime}$.
    \item If $U$ is in $\cV_{k}^\prime$ and $k^\prime < k$ then each of the sets $g^{-1}U$ and $gU$ intersects at most one set from $\cV_{k^\prime}$.
\end{enumerate}

To construct the desired sets $V_1, V_2$, and $V_3$ we now inductively define a coloring function $c(x) \in \{1, 2, 3\}$ for some points $x$ in the space. This procedure has $n$ steps and at step $k$ we ensure that for any $U \in \cV_k^\prime$ the value $c(x)$ is defined for all $x \in U$ and $c(x) = c(y)$ for any $x, y \in U$ (that is, $c$ is constant on $U$). 

Assume we are at step $k$ of the construction and let $U \in \cV_k^\prime$. If $U$ intersects some previously colored set $V \in \cV_{k^\prime}^\prime$ for $k^\prime < k$ then it is fully contained in it by item (iv) above. In this case, the function $c$ is already defined and is constant on $U$. If not, then by item (v) and the inductive assumption, $c$ takes at most one value on $g^{-1}U$ and at most one value on $gU$. Thus, we can choose a color $\sigma \in \{1, 2, 3\}$ that is distinct from either of those and set $c(x) = \sigma$ for all $x \in U$. By items (i), (ii), and (iii) this procedure can be done simultaneously for all sets $U \in \cV_k^\prime$.

Now, for $j = 1, 2, 3$ let
\[
V_j = \{x \mid c(x) = j\}.
\]
Each of these sets by construction is a union of open sets and hence open. Since $c(x) \ne c(gx)$ each of them is $\{e, g\}$-free. Clearly, these sets also have pairwise disjoint closures. Finally,
\[
V_1 \sqcup V_2 \sqcup V_3 = \bigcup_{k=1}^n \bigcup_{U \in \cV_k^\prime} U \supset \brak{\bigcup_{i=1}^nU_i^\prime} \cap \brak{\bigcup_{U \in \cV_n}U}
\]
and therefore
\begin{multline*}
[\partial V_1] + [\partial V_2] + [\partial V_3] + [X \setminus (V_1 \cup V_2 \cup V_3)] \prec 2[X \setminus (V_1 \cup V_2 \cup V_3)] \\ \prec 2\brak{\left[X \setminus \bigcup_{U \in \cV_n}U\right] + \left[X \setminus \bigcup_{i=1}^nU_i^\prime\right]} \prec 8\brak{\sum_{i=1}^n \sqbrak{\brak{\partial U_i}^{\delta}} + \sqbrak{X \setminus \bigcup_{i=1}^nU_i^{-\delta}}}.    
\end{multline*}
\end{proof}

We are now ready to prove the main theorem of this section.

\begin{theorem}\label{thm: alphabet reduction}
In Definition \ref{def: FCSB} we may assume that the number of sets in the collection $\{U_i\}_{i=1}^n$ satisfies
\[
n \le 3^{\abs{F}},
\]
independently of the marker set $V$ (for property FCSB) or the parameter $\eps$ (for property FCSB in measure). We may also assume that the sets $U_i$ have pairwise disjoint closures.
\end{theorem}

\begin{proof}
Let $F \subset G$ be a finite set containing the identity, let $V$ be a given marker set and use Lemma \ref{Lem: mult L-free subeq} to choose a marker set $U$ such that 
\[
8\abs{F}[U] \prec V.
\]
Let $\{U_i\}_{i=1}^n$ be an $F$-free cover with 
\[
\sum_{i=1}^n [\partial U_i] + \sqbrak{X \setminus \bigcup_{i=1}^nU_i} \prec U.
\]
Choose a small enough $\delta > 0$ so that
\[
\sum_{i=1}^n \sqbrak{\brak{\partial U_i}^\delta} + \sqbrak{X \setminus \bigcup_{i=1}^nU_i^{-\delta}} \prec U.
\]
For each $g \in F \setminus \{e\}$ apply Lemma \ref{lem: alphabet reduction} to obtain a collection $\cU^g = \left\{U_1^g, U_2^g, U_3^g\right\}$ of $\{e, g\}$-free sets with disjoint closures such that
\[
\sqbrak{\partial U_1^g} + \sqbrak{\partial U_2^g} + \sqbrak{\partial U_3^g} + \sqbrak{X \setminus (U_1^g \cup U_2^g \cup U_3^g)} \prec 8\brak{\sum_{i=1}^n \sqbrak{\brak{\partial U_i}^{\delta}} + \sqbrak{X \setminus \bigcup_{i=1}^nU_i^{-\delta}}} \prec 8[U].
\]
Then the collection
\[
\cU = \bigvee_{g \in F \setminus \{e\}}\cU^g
\]
has cardinality $\abs{\cU} \le 3^{\abs{F}-1}$ and consists of $F$-free sets with pairwise disjoint closures. Moreover,
\[
\sum_{W \in \cU} [\partial W] + \sqbrak{X \setminus \bigcup_{W\in \cU}W} \prec 8\abs{F}[U] \prec V,
\]
which finishes the proof.
\end{proof}

\begin{remark}\label{rem: alphabet reduction F + 1}
Theorem \ref{thm: alphabet reduction} above can in fact be further improved to conclude that $n$ can be chosen to be at most $\abs{F}+1$. However, in the present work it is only important for us that it does not depend on the marker set $V$.
\end{remark}

\section{From FCSB to URPC}
\label{sec: from FCSB to URPC}

\subsection{Equivalence between FCSB in measure and URP}
A variation of the Ornstein-Weiss algorithm allows one to prove that free actions of amenable groups on zero-dimensional spaces always have URP (\cite{KerSza20}). Importantly, the number of steps in the algorithm depends only on the prescribed invariance $(K, \delta)$ and the number of elements in the initial cover $\{U_i\}_{i=1}^n$. 

Now passing to the general case of an action with property FCSB in measure, the reduction done in Theorem \ref{thm: alphabet reduction} allows one to assume that the latter also depends only on $(K, \delta)$. Thus, we may ensure that the error coming from the boundaries remains negligible even after proceeding with the algorithm, and essentially treat the construction as if we are in a zero-dimensional setting. Note that we shall not give the full proof of Theorem \ref{thm: FCSB in measure equiv URP} and will only illustrate the effect of the boundaries described above. The rest of the argument is by now standard and we refer to \cite{KerSza20} for the full details.

\begin{lemma}[\cite{KerSza20}, Lemma 3.1]\label{lem: KS lemma}
Let $K$ be a finite subset of $G$ and let $\delta > 0$. Then there is an $\eps > 0$ such that if $F$ is a $(K, \eps)$-invariant subset of $G$ then every set $F^\prime \subset F$ with $\abs{F^\prime} \ge (1-\eps)\abs{F}$ is $(K, \delta)$-invariant.
\end{lemma}

\begin{theorem}\label{thm: FCSB in measure equiv URP}
An action has FCSB in measure if and only if it has URP.
\end{theorem}

\begin{proof}
Proposition \ref{prop: URPC implies FCSB} gives one of the implications so it remains to show that property FCSB in measure implies URP.

Let $K \subset G$ be finite and let $\delta > 0$. Choose a constant $1/2 > \eps > 0$ that satisfies the conclusion of Lemma \ref{lem: KS lemma} and such that additionally $\eps (1+\delta) < 1$. Choose $n \in \Nb$ such that $(1-\eps)^n < \eps$. Fix a $\beta > 0$ with
\[
\brak{1+\beta}^{-1}\brak{1 - \brak{1-\brak{1+\beta}\eps}^n} > 1-\eps,
\]
and choose nonempty $(K, \eps)$-invariant sets $\{e\} \in F_1 \subset F_2 \subset \ldots \subset F_n \subset G$ such that for every $1 \le i < j \le n$ the set $F_j$ is $\brak{F_i^{-1}, \beta(1-\eps)}$-invariant. Set $F = F_n^{-1}F_n$.

For brevity, denote $m = 3^{\abs{F}}$. By Theorem \ref{thm: alphabet reduction}, there exists a collection $U_1, \ldots, U_m$ of $F$-free open sets with 
\begin{equation}\label{eq: chosen sets with small boundaries}
\sum_{i=1}^m\mu(\partial U_i) + \mu\brak{X \setminus \bigcup_{i=1}^m U_i} < \frac{\eps}{\abs{F^{nm+1}}}
\end{equation}
for any $G$-invariant probability measure $\mu$. Set 
\[
\cU_i = \{U_i, X \setminus \overline{U_i}\},
\]
and for a finite subset $L \subset G$ define
\[
\cU^L = \bigvee_{i=1}^{m}\bigvee_{g \in L}g\cU_i.
\]

We now construct the Rokhlin towers using the Ornstein-Weiss algorithm. This inductive algorithm has $nm$ steps and runs through all pairs $\{(i, j) \mid 1 \le i \le n, 1 \le j \le m \}$ in lexicographical order. Set
\[
L_{i, j} = F^{(i-1)m+j-1},
\]
where $L_{1, 1} = F^0 = \{e\}$. At step $(i, j)$ we add to the castle a collection of towers such that their shapes are subsets of $F_{n-i+1}$ and their bases are elements of $\cU^{L_{i, j}}$ belonging to $U_j$ (as $\cU^{L_{i, j}}$ refines $\cU_j$ every element of it is either a subset of $U_j$ or is disjoint from it).

For the first step, we add the tower $(F_n, U_1)$ to the castle. Assume that a collection of towers $\{(S_k, V_k)\}_{k=1}^r$ has been constructed and we are now at step $(i, j) \ne (1, 1)$. Let
\[
(i^\prime, j^\prime) = 
\begin{cases}
    (i, j-1) & \mbox{if} \ j > 1 \\
    (i-1, m) & \mbox{if} \ j = 1.
\end{cases}
\]
Let $V$ be a set in $\cU^{L_{i, j}}$ belonging to $U_j$. Note that for any $g \in F$ the collection $\cU^{L_{i, j}}$ refines the collection $g\cU^{L_{i^\prime, j^\prime}}$. As both $F_{n-i+1}$ and $S_k$ are subsets of $F_n$ and $F = F_n^{-1}F_n$, it follows that for any $g \in F_{n-i+1}$, any $1 \le k \le r$, and any $h \in S_k$, the set $gV$ is either fully contained in $hV_k$ or is disjoint from it. Thus, the set 
\[
F_{n-i+1}^\prime = \{g \in F_{n-i+1} \mid gx \not \in \bigsqcup_{k=1}^rS_kV_k\}
\]
is independent from the chosen point $x \in V$. Finally, if $\abs{F_{n-i+1}^\prime} \ge (1-\eps)\abs{F_{n-i+1}}$ then we add the tower $(F_{n-i+1}^\prime, V)$ to the castle (and otherwise do nothing).

Clearly, the construction yields a collection of pairwise disjoint towers with $(K, \delta)$-invariant shapes. Moreover, let $1 \le i \le n$ and let $\{(S_k, V_k)\}_{k=1}^r$ be the castle obtained after step $(i, m)$. Then it is easy to see that for any $x \in X$ we have that
\[
\abs{F_{n-i+1}x \cap \brak{\bigsqcup_{k=1}^mS_kV_k \cup \bigcup_{j=1}^m FL_{i, m}\brak{\partial U_j} \cup F \brak{X \setminus \bigcup_{j=1}^m U_j}}} > \eps \abs{F_{n-i+1}}.
\]

Using the standard machinery (see \cite[Section 3]{KerSza20}), one shows that if $W$ is the footprint of the final castle constructed as a result of the procedure above, then
\[
\uD\brak{W \cup F^{nm+1}\brak{\bigcup_{j=1}^{m} \brak{\partial U_j} \cup \brak{X \setminus \bigcup_{j=1}^m U_j} }} > 1-\eps
\]
and thus, by \eqref{eq: chosen sets with small boundaries}, we conclude that
\[
\uD(W) > 1-2\eps. 
\]
\end{proof}

\subsection{Local tiling algorithms and URPC}

In this subsection we take an orbitwise perspective, as opposed to the spatial arguments of the previous subsection. We introduce the notion of a local algorithm, which has long been known to be the bridge between solving discrete and descriptive versions of a problem. In particular, this gives a slightly different proof of Theorem \ref{thm: FCSB in measure equiv URP}, relying on the fact that the classical Ornstein-Weiss algorithm is local (Theorem \ref{thm: OW local}). Unfortunately, this algorithm only produces a $(1-\eps)$-tiling which limits one to proving URP. The main result of this section is that if a group has a local tiling algorithm that produces an \emph{exact} tiling then one can similarly show that FCSB implies URPC (Theorem \ref{thm: FCSB and local alg imply URPC}). As we shall see in the next subsection, many amenable groups admit exact local tiling algorithms. 

\begin{definition}\label{def: local map}
Let $Q, R$ be arbitrary sets and let $W \subset G$ be a finite set. Let $\cF \subset Q^G$ be a \emph{right}-invariant subset of maps from $G$ to $Q$ and consider it's restriction $\left.\cF\right|_W$ to the set $W$. Let $\phi : \left.\cF\right|_W \to R$ be an arbitrary map. For any $x \in \cF$ we can define a function $\hat{\phi}_x : G \to R$ by
\[
\hat{\phi}_x(g) = \phi\brak{\brak{x_{hg}}_{h \in W}}.
\]
Thus the assignment $x \mapsto \hat{\phi}_x$ is a map from $\cF$ to $R^G$. 

Given an arbitrary function $\Phi : \cF \to R^G$ (which we usually write as a collection of maps $\Phi_x : G \to R$ for $x \in \cF$) we will say that it is \emph{local} (alternatively, \emph{depends locally on $\cF$} or \emph{depends locally on $x$}) if there exists some $W$ and $\phi$ as above such that $\Phi = \hat{\phi}$. In this case we will additionally say that $\Phi$ is \emph{determined by the window $W$}. Note that any local function is equivariant with respect to the right action of $G$ on the shifts.

In the special case when $R = \{0, 1\}$ we will identify the space $R^G$ with the space of subsets of $G$ in the usual way. Thus, given a collection of subsets $A_x, x \in \cF$ we say that it depends locally on $\cF$ if the assignment $x \mapsto \mathbbm{1}_{A_x}$ is local.
\end{definition}

\begin{remark}\label{rem: local join composition}
We start by recording the following easy observations.
\begin{itemize}
    \item Let $\Phi : \cF \to R^G$ and $\Psi : \cF \to P^G$ be local maps. Then the join map $(\Phi, \Psi) : \cF \to (R \times P)^G$ is local as well.
    \item Let $\Phi : \cF \to R^G$ and $\Psi : \mathrm{Im}(\Phi) \subset R^G \to P^G$ be local maps. Then the composition $\Psi \circ \Phi : \cF \to P^G$ is local as well.
\end{itemize}
\end{remark}

In the Definition \ref{def: local map} above the right-invariant set $\cF$ should be considered as a space of possible input data. We now introduce the main family of subshifts that will be used as the data $\cF$.

\begin{definition}
Let $F$ be a finite symmetric subset of $G$ containing the identity $e$ and set $n = n(F) = 3^{\abs{F}}$. Define 
\[
\cC_F = \left\{\brak{x_g}_{g \in G} \in \{1, \ldots, n\}^G \mid x_h \ne x_{gh} \ \mbox{for any} \ g \in F \setminus \{e\} \ \mbox{and any} \ h \in G\right\}.
\]
In other words, $\cC_F$ is the space of regular $n$-colorings of the graph $\mathrm{Cay}(G, F)$. 
\end{definition}

\begin{remark}\label{rem: coloring extension}
Let $W \subset G$ be any subset of $G$. Then the restriction $\left.\cC_F\right|_W$ is the space of regular $n$-colorings of $\left.\mathrm{Cay}(G, F)\right|_W$, that is the set
\[
\left\{\brak{x_g}_{g \in W} \in \{1, \ldots, n\}^W \mid x_h \ne x_{gh} \ \mbox{for any} \ g \in F \setminus \{e\} \ \mbox{and any} \ h \in W \ \mbox{such that} \ gh \in W\right\}.
\]
Indeed, it is clear that any element in the restriction $\left.\cC_F\right|_W$ belongs to the set above. Conversely, a simple greedy coloring algorithm shows that any regular $n$-coloring of $\left.\mathrm{Cay}(G, F)\right|_W$ extends to a regular $n$-coloring of $\mathrm{Cay}(G, F)$.
\end{remark}

\begin{lemma}\label{lem: local coloring F prime F}
Let $F, F^\prime \subset G$ be finite symmetric subsets of $G$ containing the identity such that $F \subset F^\prime$. Then there is a local map from $\cC_{F^\prime}$ to $\cC_F$.
\end{lemma}

\begin{proof}
Let $x \in \cC_{F^\prime}$ be given. As usual, set $n = 3^{\abs{F}}$. We inductively define a function $y : G \to \{1, 2 \ldots, n\}$ on each color class of $x$ (that is, the sets $\{g \in G \mid x(g) = k\}$ for $k = 1, 2, \ldots, 3^{\abs{F^\prime}}$) such that at each step the new values $y(g)$ depend only the local data of $x$ and previously defined values of $y$. Since the number of steps are finite, this ensures that the final assignment $x \mapsto y$ is local.

The construction is again a simple greedy algorithm. Specifically, assume that $y(g)$ has been defined for all $g$ with $x(g) < k$ and we are now at step $k$. For any $g \in G$ with $x(g) = k$ set
\[
y(g) = \min(\{1 \le i \le n \mid y(hg) \ne i \ \mbox{for any} \ h \in F \ \mbox{such that} \ x(hg) < k\}).
\]
It is easily seen that this procedure results in an element $y \in \cC_F$.
\end{proof}

\begin{remark}
In the proof above $y$ clearly only takes values in $\{1, 2, \ldots, \abs{F}+1\}$. This should be compared to Remark \ref{rem: alphabet reduction F + 1}. Once again, for our purposes the exact number is not important. 
\end{remark}

\begin{definition}\label{def: tiling}
Let $G$ be an amenable group. An \emph{approximate tiling} of $G$ is a map $S : G \to 2^D \sqcup \{\star\}$, where $D$ is some finite subset of $G$, such that 
\[
S(g)g \cap S(g^\prime)g^\prime = \emptyset
\]
whenever $g, g^\prime \in G$ are distinct and $S(g), S(g^\prime) \ne \star$.

We will refer to the elements $g \in G$ with $S(g) \ne \star$ as \emph{centers} and to the set $C(S) = \{g \in G \mid S(g) \ne \star\}$ as the \emph{center set}. A pair $(S(g), g)$, where $g \in C(S)$, is called a \emph{tile} and the set $S(g)$ is its \emph{shape}. An approximate tiling gives a partition of the group
\[
G = \bigsqcup_{g \in C(S)}S(g)g \sqcup \brak{G \setminus \bigsqcup_{g \in C(S)}S(g)g}.
\]
The set $G \setminus \bigsqcup_{g \in C(S)}S(g)g$ is called the \emph{remainder} of the approximate tiling. We additionally define the \emph{relative position} function $T : G \to D \sqcup \{\star\}$ by
\[
T(g) =
\begin{cases}
h \in G & \mbox{if} \ h^{-1}g \in C(S) \ \mbox{and} \ h \in S(h^{-1}g) \ (\mbox{and hence} \ g \in S(h^{-1}g)h^{-1}g), \\
\star & \mbox{if no such} \ h \ \mbox{exists}.
\end{cases}
\]

We say that the approximate tiling $S$ is
\begin{enumerate}
    \item \emph{$(K, \eps)$-invariant} if for any $g \in C(S)$ the shape $S(g)$ is a $(K, \eps)$-invariant subset of $G$,
    \item \emph{$(1-\eps)$-covering relative to a finite subset $L \subset G$} if 
    \[
    \uD_L\brak{G \setminus \bigsqcup_{g \in C(S)}S(g)g} < \eps,
    \]
    \item \emph{exact} if $G = \bigsqcup_{g \in C(S)}S(g)g$.
\end{enumerate}
\end{definition}

\begin{remark}
Note that the center set, the remainder set, and the relative position function all depend locally on $S$. Thus, by Remark \ref{rem: local join composition}, if a family of tilings $S_x$ depend locally on some input data $x \in \cF$ then so do each of the above.
\end{remark}

The classic Ornstein-Weiss algorithm produces a $(K, \eps)$-invariant $(1-\eps)$-covering tiling based on a $F$-regular coloring in a local way. This can be deduced, for instance, from it's description in \cite{DowHucZha19}. We remark that the main result of that paper is that every amenable group admits an exact tiling -- however, not necessarily one that depends locally on a coloring.

\begin{theorem}\label{thm: OW local}
Let $G$ be an amenable group, let $K \subset G$ be finite, and let $\eps > 0$. Then there exist finite (symmetric) sets $F, D, L \subset G$ and a local function 
\[
S : \cC_F \to \brak{2^D \sqcup \{\star\}}^G
\]
such that for any $x \in \cC_F$ the element $S_x$ is a $(K, \eps)$-invariant $(1-\eps)$-covering relative to the set $L$ approximate tiling.
\end{theorem}

For any free action of $G$ on a measure space, Borel space, or zero-dimensional topological space, one can properly color the Shreier graph ${\mathrm{Shr}}(X, F)$ into $n=3^{\abs{F}}$ colors so that each color class is measurable, Borel, or clopen respectively. This is equivalent to saying that there is a partition 
\[
X = \bigsqcup_{i=1}^nU_i
\]
of the space into $F$-free measurable, Borel, or clopen subsets $U_i$. This coloring thus gives a natural map from the space $X$ to the subshift $\cC_F$ by
\[
x \mapsto (c(gx))_{g \in G},
\]
where $c : X \to \{1, 2, \ldots, n\}$ is the coloring function defined by setting $c \equiv i$ on $U_i$. By composing the map into the subshift with a local algorithm that solves a certain discrete problem, one can then obtain a measurable, Borel, or clopen solution. For instance, composing with the tiling function $S$ given by Theorem \ref{thm: OW local} recovers the original Ornstein-Weiss theorem for actions on measure spaces, as well as its Borel and zero-dimensional versions.

Although we are not in the zero-dimensional setting, the conclusion of Theorem \ref{thm: alphabet reduction} can be regarded as an analogy to the coloring described above --- the space can be partitioned into $n$ open $F$-free sets up to an arbitrarily small remainder. Naturally, we do not have the map from $X$ to the subshift $\cC_F$ anymore as some points remain uncolored. However, if we want to apply a local algortihm determined by some window $W$ to a point $x \in X$, we do not need to look at the whole orbit of $x$ --- it is sufficient for the set $Wx$ to be colored. As the window $W$ is independent of the smallness of the remainder, we may choose the latter so that most of the points in the space possess this property. We refer to the proof of Theorem \ref{thm: FCSB and local alg imply URPC} below for the details of this argument.

Note that, in particular, applying the algorithm given by Theorem \ref{thm: OW local} gives a slightly different proof of Theorem \ref{thm: FCSB in measure equiv URP} from the previous subsection (here we can estimate $W \subset F^{nm+1}$, which is exactly the set that appears in the proof). What this perspective shows, however, is that the precise algorithm used was not important, only the fact that it was local.

\begin{definition}\label{def: local tiling alg}
A group $G$ has a \emph{local tiling algorithm} if for any finite $K \subset G$ and $\eps > 0$ there exist finite (symmetric) sets $F, D \subset G$ and a local function
\[
S : \cC_F \to \brak{2^D \sqcup \{\star\}}^G
\]
such that for any $x \in \cC_F$ the map $S_x$ is a $(K, \eps)$-invariant exact tiling of $G$.
\end{definition}

One feature of local maps is that we can apply it even when the input data is incomplete to obtain a partial structure. We formalize this intuition in the following (slightly cumbersome) lemma.

\begin{lemma}\label{lem: partial tiling}
Let $G$ be a group with a local tiling algorithm. Let $K \subset G$ be finite, let $\eps > 0$, and let $F, D, S$ be as in Definition \ref{def: local tiling alg}. Since $S$ is a local function, there is a finite window $W$ and a map $\phi : \left.\cC_F\right|_{W} \to 2^D \sqcup \{\star\}$ such that 
\[
S_x(g) = \phi\brak{(x_{hg})_{h \in W}}.
\]

Now, let $A \subset G$ be an arbitrary set and let $c \in \left.\cC_F\right|_{A}$ be a proper $F$-coloring of $A$. Consider the set of centers $C_A = \{g \in G \mid Wg \subset A \ \mbox{and} \ \phi\brak{(c_{hg})_{h \in W}} \ne \star\}$ and the corresponding collection of tiles $\{(\phi\brak{(c_{hg})_{h \in W}}, g)\}_{g \in C_A}$. Then the sets $\phi\brak{(c_{hg})_{h \in W}}g$ are pairwise disjoint for $g \in C_A$ and
\[
A \setminus \brak{DW^{-1}(G \setminus A)} \subset \bigsqcup_{g \in C_A}\phi\brak{(c_{hg})_{h \in W}}g.
\]
\end{lemma}

\begin{proof}
Let $x \in \cC_F$ be an arbitrary extension of $c$ to the whole group $G$ (see Remark \ref{rem: coloring extension}). Note that $\{(\phi\brak{(c_{hg})_{h \in W}}, g)\}_{g \in C_A}$ is a subcollection of tiles in the tiling $S_x$, hence the sets $\phi\brak{(c_{hg})_{h \in W}}g$ are pairwise disjoint. Moreover, we have that the set of remaining centers $C(S_x) \setminus C_A$ is clearly a subset of $W^{-1}(G \setminus A)$, thus
\[
\bigsqcup_{g \in C(S_x) \setminus C_A} \phi\brak{(x_{hg})_{h \in W}}g \subset DW^{-1}(G \setminus A).
\]
Since the tiling is exact, the second claim follows.
\end{proof}

We can now prove the main result in this subsection, which shows that a local tiling algorithm is a bridge between property FCSB and URPC.

\begin{theorem}\label{thm: FCSB and local alg imply URPC}
Suppose that an amenable group $G$ has a local tiling algorithm. Let $G \act X$ be an action that has property FCSB. Then it has URPC.
\end{theorem}

\begin{proof}
Let $K \subset G$ be a finite set, let $\eps > 0$, and let $U$ be a marker set. Since $G$ has a local tiling algorithm there exist finite symmetric sets $F, D \subset G$ and a local function
\[
S : \cC_F \to \brak{2^D \sqcup \{\star\}}^G
\]
such that for any $x \in \cC_F$ the element $S_x$ is a $(K, \eps)$-invariant exact tiling of $G$. By definition of a local function, there is a finite window $W \subset G$ and a function $\phi : \left.\cC_F\right|_{W} \to 2^D \sqcup \{\star\}$ such that 
\[
S_x(g) = \phi\brak{(x_{hg})_{h \in W}}.
\]
Recall (Remark \ref{rem: coloring extension}) that $\left.\cC_F\right|_{W}$ is simply the space of $n$-colorings of the set $W$ (recall that $n = 3^{\abs{F}}$).

Use Theorem \ref{thm: alphabet reduction} and Lemma \ref{Lem: mult L-free subeq} to find an $F$-free disjoint collection $\{U_i\}_{i=1}^n$ such that
\[
\abs{DW^{-1}}\sqbrak{X \setminus \bigsqcup_{i=1}^n U_i} \prec U.
\]
Define a partial coloring function $c : X \to \{1, 2, \ldots, n\} \sqcup \{\star\}$ by 
\[
c(y) = 
\begin{cases}
i & \mbox{if} \ y \in U_i, \\
\star & \mbox{if} \ y \in X \setminus \bigsqcup_{i=1}^nU_i.
\end{cases}
\]

Now, let $\mathrm{x} \in \left.\cC_F\right|_{W}$ be a configuration such that $\phi(\mathrm{x}) \ne \star$ (i.e., $\mathrm{x}$ is central) and define
\[
V(\mathrm{x}) = \{y \in X \mid c(gy) = \mathrm{x}(g) \ \mbox{for every} \ g \in W\}.
\]
Consider the collection $\{(\phi(\mathrm{x}), V(\mathrm{x}))\}$, where $\mathrm{x}$ ranges over all central configurations in $\left.\cC_F\right|_{W}$. We claim that this is a castle that witnesses URPC for the action $\GonX$. 

Indeed, first note that it is clear that the bases $V(\mathrm{x})$ are open and shapes $\phi(\mathrm{x})$ are $(K, \eps)$-invariant. Now, let $y \in X$ be an arbitrary point and identify the orbit $Gy$ with the group $G$ in the obvious way. For each central configuration $\mathrm{x} \in \left.\cC_F\right|_{W}$ let $C_\mathrm{x} \subset G$ be the set of elements $g$ such that $gy \in V(\mathrm{x})$. Then the collection $\{(\phi(\mathrm{x}), g)\}$, where $\mathrm{x}$ ranges over all central configurations and $g \in C_\mathrm{x}$, is precisely the partial tiling described in Lemma \ref{lem: partial tiling} with
\[
A = \{g \in G \mid c(gy) \ne \star\}.
\]
It then follows that each pair $(\phi(\mathrm{x}), V(\mathrm{x}))$ is a tower, that all the levels of all towers are pairwise disjoint, and that the remainder of the tiling is contained in the set
\[
DW^{-1}\brak{X \setminus \bigcup_{i=1}^n U_i} \prec \abs{DW^{-1}}\sqbrak{X \setminus \bigsqcup_{i=1}^n U_i} \prec U.
\]

\end{proof}

\subsection{Groups that admit local tiling algorithms}
For actions of amenable groups on zero-dimensional spaces, the existence of exact tilings by Rokhlin towers has been well-studied (\cite{Ker20}, \cite{DowZha23}, \cite{KerSza20}, \cite{KerNar21}, \cite{Nar24}, \cite{NarPet24}). As it turns out, most of the theory can be translated into the language of local algorithms. In this section we will sketch the translations of a few central results. The remaining ones can be obtained from the literature in a similar manner. 

We begin by introducing an appropriate notion of comparison that later turns out to be equivalent to having a local tiling algorithm.

\begin{definition}
Let $A, B \subset G$, let $D \subset G$ be finite, and let $R : A \to D$. We say that $R$ \emph{implements the subequivalence} between $A$ and $B$ if the map $g \mapsto R(g)g$ is an injection from $A$ to $B$. If it is moreover a bijection, then we say that $R$ \emph{implements the equivalence} between $A$ and $B$. The latter is easily seen to be an equivalence relation.

More generally, given elements $a = (a_g)_{g \in G}, b = (b_g)_{g \in G} \in \{0, 1, \ldots, m\}^G$ and a function $R$ on $G$ such that $R(g) = (R(g)_1, R(g)_2, \ldots, R(g)_{a_g}) \in D^{a_g}$, we say that $R$ \emph{implements the subequivalence} between $a$ and $b$ if for any $g \in G$ we have
\begin{equation}\label{eq: multiset subeq}
\abs{\{(i, h) \in \{1, \ldots, m\} \times G \mid i \le a_h \ \mbox{and} \ R(h)_ih = g\}} \le b_g.
\end{equation}
Again, if the inequality \eqref{eq: multiset subeq} is in fact an equality for all $g \in G$, then we say that $R$ \emph{implements the equivalence} between $a$ and $b$. Once again, this is indeed an equivalence relation.

Note that if $R$ implements the equivalence between $A$ and $B$ (resp. $a$ and $b$) then the ``inverse'' map that implements the equivalence between $B$ and $A$ (resp. $b$ and $a$) depends locally on $R$.
\end{definition}

\begin{definition}
Let $L \subset G$ be finite and $\alpha, \beta > 0$. Define
\[
\mathcal{UD}_{L, \alpha} = \left\{\brak{x_g}_{g \in G} \in \{0, 1\}^G \mid \sum_{g \in L}x_{gh} < \alpha\abs{L} \ \mbox{for any} \ h \in G\right\}
\]
and
\[
\mathcal{LD}_{L, \beta} = \left\{\brak{x_g}_{g \in G} \in \{0, 1\}^G \mid \sum_{g \in L}x_{gh} > \beta\abs{L} \ \mbox{for any} \ h \in G\right\}.
\]
That is, $\mathcal{UD}_{L, \alpha}$ is the space of subsets $A \subset G$ with $\uD_L(A) < \alpha$ and $\mathcal{LD}_{L, \beta}$ is the space of subsets $B \subset G$ with $\lD_L(B) > \beta$.

For a number $m \in \Nb$ we can similarly define
\[
\mathcal{UD}^m_{L, \alpha} = \left\{\brak{x_g}_{g \in G} \in \{0, 1, \ldots, m\}^G \mid \sum_{g \in L}x_{gh} < \alpha\abs{L} \ \mbox{for any} \ h \in G\right\}
\]
and
\[
\mathcal{LD}^m_{L, \beta} = \left\{\brak{x_g}_{g \in G} \in \{0, 1, \ldots,  m\}^G \mid \sum_{g \in L}x_{gh} > \beta\abs{L} \ \mbox{for any} \ h \in G\right\}.
\]
These can be viewed as spaces of $m$-tuples of subsets of $G$ with the restriction on the joint density.
\end{definition}

\begin{definition}
A group $G$ has a \emph{local comparison algorithm} if for any finite $L \subset G$ and any pair $0 \le \alpha < \beta \le 1$ there are finite (symmetric) sets $F, D \subset G$ and a local function $R$ of $\cC_F \times \mathcal{UD}_{L, \alpha} \times \mathcal{LD}_{L, \beta}$ such that if $(x, A, B) \in \cC_F \times \mathcal{UD}_{L, \alpha} \times \mathcal{LD}_{L, \beta}$ then $R_{(x, A, B)}$ implements the subequivalence between $A$ and $B$.

We say that $G$ has a \emph{local multiset comparison algorithm} if for any $m \in \Nb$, any finite $L \subset G$ and any pair $0 \le \alpha < \beta \le m$ there are finite sets $F, D \subset G$ and a local function $R$ of $\cC_F \times \mathcal{UD}^m_{L, \alpha} \times \mathcal{LD}^m_{L, \beta}$ such that $R_{(x, a, b)}$ implements the subequivalence between $a$ and $b$.
\end{definition}

The following is a basic exercise about the F{\o}lner sets. We omit the proof.

\begin{lemma}\label{lem: amenable delta kappa}
Let $G$ be an amenable group, let $K \subset G$ be finite, and let $\eps > 0$. Then there are $\kappa, \delta > 0$ such that if $S \subset G$ is $(K, \delta)$-invariant then
\begin{enumerate}
    \item $\abs{S} > 1/\kappa$ (hence $\lfloor \kappa\abs{S} \rfloor > \kappa\abs{S}/2$) and
    \item if $S^\prime \supset S$ is any set with $\abs{S^\prime} \le (1+\kappa)\abs{S}$ then $S^\prime$ is $(K, \eps)$-invariant.
\end{enumerate}
\end{lemma}

\begin{theorem}\label{thm: local equiv prop}
Let $G$ be an amenable group. The following are equivalent.
\begin{enumerate}
    \item It has a local tiling algorithm.
    \item It has a local comparison algorithm.
    \item It has a local multiset comparison algorithm.
\end{enumerate}
\end{theorem}

\begin{proof}
It is clear that (iii) implies (ii). 

\medskip

\textbf{(ii) $\Longrightarrow$ (i)}. Assume that $G$ has a local comparison algorithm, let $K \subset G$ be finite and let $\eps > 0$. Let $1/2 > \kappa, \delta > 0$ be the constants provided by Lemma \ref{lem: amenable delta kappa}. By Theorem \ref{thm: OW local}, there are sets $F, D, L \subset G$ and a local function 
\[
S : \cC_F \to \brak{2^D \sqcup \{\star\}}^G
\]
that gives a $(K, \delta)$-invariant approximate tiling that is $(1 - \kappa/4)$-covering relative to $L$. 

For each $(K, \delta)$-invariant set $Q \subset D$ pick a subset $B(Q) \subset Q$ with 
\[
\abs{B(Q)} = \lfloor \kappa\abs{Q} \rfloor > \kappa\abs{Q}/2.
\]
Let $x \in \cC_F$, let $C(S_x)$ be the center set of the associated tiling and define a subset $B_x \subset \bigsqcup_{g \in C(S_x)}S_x(g)g$ by
\[
B_x \cap S_x(g)g = B(S_x(g))g
\]
for any $g \in C(S_x)$. In other words, $B_x$ picks a $\kappa$-proportion of elements out of every tile in a deterministic way and therefore depends locally on $x$. Additionally, let $A_x$ be the remainder of the tiling
\[
A_x = G \setminus \bigsqcup_{g \in C(S_x)}S_x(g)g.
\]
There is a sufficiently big F{\o}lner set $L^\prime$ that depends only on $D$ and $L$ such that
\[
\lD_{L^\prime}(B_x) > (1-\kappa/4)\kappa/2 > \kappa/3 \quad \mbox{and} \quad \uD_{L^\prime}(A_x) < \kappa/4.
\]

By assumption, $G$ has a local comparison algorithm hence there exist a finite sets $F^\prime$ containing $F$, a finite set $D^\prime$, and a local function $R$ of $\cC_{F^\prime} \times \mathcal{UD}_{L^\prime, \kappa/4} \times \mathcal{LD}_{L^\prime, \kappa/3}$ such that $R_{(x, A, B)}$ implemenets the subequivalence between $A$ and $B$. By Lemma \ref{lem: local coloring F prime F}, we may view the tiling $S$ as a local function of $\cC_{F^\prime}$. In particular, sets $A_x$ and $B_x$ constructed above then also depend locally on $\cC_{F^\prime}$ and therefore by Remark \ref{rem: local join composition}, the function $R_{(x, A_x, B_x)}$ is a local function of $\cC_{F^\prime}$ as well.

Fix an $x \in \cC_{F^\prime}$ and for each tile $(S_x(g), g)$, where $g \in C(S_X)$, consider the set $A_{x, g} = \{h \in A_x \mid R_{(x, A_x, B_x)}(h)h \in S_x(g)g\}$ of points in $A_x$ that are mapped into $S_x(g)g$ by the injection $R_{(x, A_x, B_x)}$. Define a new tiling $S^\prime$ depending on $\cC_{F^\prime}$ by setting
\[
S^\prime_x(g) = 
\begin{cases}
\star & \mbox{if} \ S_x(g) = \star, \\
S_x(g) \sqcup A_{x, g}g^{-1} & \mbox{otherwise}.
\end{cases}
\]
It is not difficult to see that $S^\prime$ is an exact tiling and is again a local function of $\cC_{F^\prime}$. Moreover, as $\abs{A_{x, g}} \le \kappa\abs{S_x(g)}$, by Lemma \ref{lem: amenable delta kappa} we have that the new shapes are $(K, \eps)$-invariant which finishes the proof.

\medskip

\textbf{(i) $\Longrightarrow$ (iii)}. Let $L \subset G$, let $m \in \Nb$, and let $0 \le \alpha < \beta \le m$ be given. There is a finite set $K \subset G$ and $\eps > 0$ such that if $a \in \mathcal{UD}^m_{L, \alpha}$, $b \in \mathcal{LD}^m_{L, \beta}$, and $Q$ is any $(K, \eps)$-invariant subset of $G$ then
\[
\sum_{g \in Q}a_g < \sum_{g \in Q}b_g.
\]

By assumption, there are finite sets $F, D \subset G$ and a $(K, \eps)$-invariant exact tiling $S$ that depends locally on $\cC_F$. For each $(K, \eps)$-invariant subset $Q \subset D$ there are finitely many possible pairs of elements $\mathrm{a} \in \{0, 1, \ldots, m\}^Q$ and $\mathrm{b} \in \{0, 1, \ldots, m\}^Q$ such that
\[
\sum_{g \in Q}\mathrm{a}_g < \sum_{g \in Q}\mathrm{b}_g.
\]
For each such pair choose a function $\mathrm{R}_{(Q, \mathrm{a}, \mathrm{b})}$ on $Q$ such that 
\[
\mathrm{R}_{(Q, \mathrm{a}, \mathrm{b})}(g) = (\mathrm{R}_{(Q, \mathrm{a}, \mathrm{b})}(g)_1, \mathrm{R}_{(Q, \mathrm{a}, \mathrm{b})}(g)_2, \ldots, \mathrm{R}_{(Q, \mathrm{a}, \mathrm{b})}(g)_{\mathrm{a}_g}) \in \brak{D^{-1}D}^{\mathrm{a}_g}
\]
and
\[
\abs{\{(i, h) \in \{1, \ldots, m\} \times Q \mid i \le \mathrm{a}_h \ \mbox{and} \ \mathrm{R}_{(Q, \mathrm{a}, \mathrm{b})}(h)_ih = g\}} \le \mathrm{b}_g.
\]

Now, fix elements $x \in \cC_F$, $a \in \mathcal{UD}^m_{L, \alpha}$, and $b \in \mathcal{LD}^m_{L, \beta}$. Let $g \in G$ and let $h \in C(S_x)$ be such that $g \in S_x(h)h$ (which exists since the tiling $S$ is exact). Let $\mathrm{a}, \mathrm{b} \in \{0, 1, \ldots, m\}^{S_x(h)}$ be given by
\[
\mathrm{a}(t) = a(th) \quad \mbox{and} \quad \mathrm{b}(t) = b(th).
\]
Define
\[
R_{(x, a, b)}(g)_i = \mathrm{R}_{(S_x(h), \mathrm{a}, \mathrm{b})}(gh^{-1})_i
\]
It is straightforward to check that $R$ depends locally on $x, a$, and $b$ and that it implements the subequivalence between $a$ and $b$.
\end{proof}

We now give some examples of groups that have local tiling algorithms. First of all, it is not difficult to see that the following lemma holds.

\begin{lemma}
The class of amenable groups with local tiling algorithms is closed under taking direct limits.
\end{lemma}

The following theorem is due to Downarowicz and Zhang \cite{DowZha23}. We remark that in fact they prove a stronger result --- there exists a comparison algorithm that depends locally only on the sets $A$ and $B$ (and does not require an element in $\cC_F$). This corresponds to the difference between considering all actions of $G$ versus only free actions of $G$.

\begin{theorem}
Groups of locally subexponential growth have local comparison algorithms (and hence also local tiling algorithms by Theorem \ref{thm: local equiv prop}).
\end{theorem}

The next theorem is a translation of the main result in \cite{Nar24} into the language of local algorithms.

\begin{theorem}\label{thm: local tiling extension}
If $H \lhd G$ is an infinite normal subgroup with a local tiling algorithm then $G$ also has a local tiling algorithm.
\end{theorem}

\begin{proof}
By Theorem \ref{thm: local equiv prop} it is enough to show that $G$ has a local comparison algorithm. Let $L \subset G$ be finite and let $0 \le \alpha < \beta \le 1$. Set $\gamma = \beta - \alpha > 0$ and choose $K \subset H$ and $\eps > 0$ so that any $(K, \eps)$-invariant subset of $H$ has cardinality at least $3\abs{L}/\gamma$. Since $H$ has a local tiling algorithm, there is a finite symmetric set $F \subset H$ and a $(K, \eps)$-invariant exact tiling $S$ that depends locally on $\cC_F$. For any shape $Q$ occurring in the tiling $S$ and any subset $\mathrm{C} \subset Q$ choose an arbitrary partition 
\[
\mathrm{C} = \bigsqcup_{g \in L}\mathrm{C}_g
\]
such that 
\[
\abs{\mathrm{C}_g} \in \left\{\lceil\frac{\abs{C}}{\abs{L}}\rceil, \lfloor\frac{\abs{C}}{\abs{L}}\rfloor\right\}
\]
for any $g \in L$. 

We remark that the group $G$ can be written as a disjoint union of cosets of $H$ and an $F$-proper coloring of $G$ is simply a collection of $F$-proper colorings for each coset of $H$. Hence, given an element $x \in \cC_F$, we can apply $S$ to get an exact $(K, \eps)$-invariant tiling of each coset of $H$ and thus of the whole group $G$. Employing a slight abuse of notation we will denote this tiling by $S$ as well.

Let $C \subset G$ be an arbitrary set, let $x \in \cC_F$, and let $g \in C$. Since $S_x$ is an exact tiling of $G$, there is some $h \in C(S_x)$ such that $g \in S_x(h)h$. Set $\mathrm{C} = (C \cap S_x(h)h)h^{-1} \subset S_x(h)$ and define
\[
\tilde{R}_{x, C}(g) = r^{-1} \in L^{-1} \quad \mbox{if} \quad gh^{-1} \in \mathrm{C}_r.
\]
The function $\tilde{R}$ depends locally on $x$ and $C$ and implements an equivalence between $C$ and some element $c \in \{0, 1, \ldots, \abs{L}\}^G$.

Let $A \in \mathcal{UD}_{L, \alpha}$, let $B \in \mathcal{LD}_{L, \beta}$, let $x \in \cC_F$, and let $a, b \in \{0, 1, \ldots, \abs{L}\}^G$ be obtained using $\tilde{R}$ from $A$ and $B$ respectively. Thus, $\tilde{R}_{x, A}$ implements the equivalence between $A$ and $a$ and there exists a function which we denote by $\tilde{R}_{x, b}$ (the ``inverse'' of $\tilde{R}_{x, B}$) that implements the equivalence between $b$ and $B$. One may check that, by construction, there exists a sufficiently big F{\o}lner set $L^\prime \subset H$ (independent of $A$, $B$, and $x$) such that
\[
\sum_{g \in L^{\prime}}a_{gh} < (\alpha + \gamma/3)\abs{L^\prime} < (\alpha + 2\gamma/3)\abs{L^\prime} < \sum_{g \in L^{\prime}}b_{gh}
\]
for any $h \in G$. Since $H$ has a local multiset comparison algorithm by Theorem \ref{thm: local equiv prop}, there is a local function $R$ such that $R_{(x, a, b)}$ implements the subequivalence between $a$ and $b$ (here we again enlarged the set $F$ using Lemma \ref{lem: local coloring F prime F}). Finally, the subequivalence between $A$ and $B$ is obtained by composing
\[
\tilde{R}_{x, b} \circ R_{(x, a, b)} \circ \tilde{R}_{x, A}.
\]
\end{proof}

\section{Mean dimension and shift embeddability for actions of amenable groups}
\label{sec: mdim and shift embed}

\subsection{Encoding URPC covers}\label{ssec: encoding URPC}

\begin{definition}
Let $F : X \to Z$ be a function and let $U \subset X$ be a set. We say that $F$ \emph{encodes} $U$ if 
\[
F^{-1}(F(U)) = U.
\]
Equivalently, for every pair of points $x, y \in X$ such that $x \in U$ and $y \not \in U$, we have
\[
F(x) \ne F(y).
\]
In other words, by knowing the value $F(x)$ we can determine whether it is in $U$ or not.

We say that $F$ encodes a collection of sets $\{U_i\}_{i \in I}$ if it encodes each of the sets $U_i$ in the collection.
\end{definition}

\begin{definition}
Let $f : X \to Y$ be some function on $X$. We associate to it a function $I_f : X \to Y^G$ defined by
\[
I_f(x) = \brak{f(gx)}_{g \in G}.
\]
If $f$ is continuous then so is $I_f$.
\end{definition}

We collect some easy but important observations.

\begin{remark}\label{rem: encoding shift set}
If $I_f$ encodes $U$ then it also encodes $gU$ for every $g \in G$.
\end{remark}

\begin{corollary}\label{cor: encoding element recovers URPC}
Let $K \subset G$ be finite, let $\eps > 0$, let $\brak{\{S_i\}_{i=1}^n, \{V_i\}_{i=1}^n, \{g_j\}_{j=1}^m, \{U_j\}_{j=1}^m}$ be a $(K, \eps)$-URPC witness, and let $f : X \to Y$ be a function. If $I_f$ encodes the collection
\[
\{V_i\}_{i=1}^n \sqcup \{U_j\}_{j=1}^m
\]
then it encodes the whole URPC cover 
\[
\{gV_i\}_{1 \le i \le n, g \in S_i} \sqcup \{U_j\}_{j=1}^m.
\]
\end{corollary}

\begin{remark}\label{rem: encoding set and collection}
Let $U$ be a set and suppose that 
\[
U = \bigcup_{i=1}^kU_i
\]
for some finite collection of sets $U_i$. If $I_f$ encodes the collection $\{U_i\}_{i=1}^k$ then it also encodes the set $U$. By Remark \ref{rem: encoding shift set}, this assumption is equivalent to $I_f$ encoding the shifted collection $\{g_iU_i\}_{i=1}^k$ for any choice of elements $g_i \in G$.
\end{remark}

We now prove the main result of this subsection.

\begin{lemma}\label{lem: encoding URPC 2-dim}
Let $G \act X$ be an action with URPC and let $U$ be a marker set. Then there exists a continuous function $\phi = (\phi_1, \phi_2) : X \to [0, 1]^2$ supported on $U$ such that for any finite $K \subset G$ and $\eps > 0$ the map $I_\phi$ encodes some URPC cover for $K$ and $\eps$.

Moreover, $\phi$ can be chosen so that the open support of both $\phi_1$ and $\phi_2$ is equal to $U$ and so that $\phi_1(x) > \phi_2(x)$ for any $x \in U$. 
\end{lemma}

\begin{proof}
Choose an increasing sequence of finite set $K_m$ with $\bigcup_{m \in \Nb}K_m = G$ and a sequence of numbers $\eps_m > 0$ with $\eps_m \to 0$. Use Lemma \ref{lem: eps levels subeq marker} to find a $(K_1, \eps_1)$-URPC witness, say $\brak{\{S_i\}_{i=1}^n, \{V_i\}_{i=1}^n, \{g_j\}_{j=1}^m, \{U_j\}_{j=1}^m}$, such that its encoding element $E$ satisfies
\[
E = \sum_{i=1}^n[V_i] + \sum_{j=1}^m[U_j] \prec U.
\]
By Remark \ref{rem: URPC cover stable}, we can find some $\delta > 0$ such that for any open sets $V_i^\prime$ and $U_j^\prime$ with $V_i^{-\delta} \subset V_i^\prime \subset V_i$ and $U_j^{-\delta} \subset U_j^\prime \subset U_j$, the data $\brak{\{S_i\}_{i=1}^n, \{V_i\}_{i=1}^n, \{g_j\}_{j=1}^m, \{U_j\}_{j=1}^m}$ is still a a $(K_1, \eps_1)$-URPC witness. Note that we clearly have that
\[
\sum_{i=1}^n\sqbrak{\, \overline{V_i^{-\delta}}\, } + \sum_{j=1}^m\sqbrak{\, \overline{U_j^{-\delta}}\, } \prec U.
\]
Now, using the definition of subequivalence of sets, Remark \ref{rem: subeq cover properties}, and Remark \ref{rem: encoding set and collection}, we can find some finite collection $\{W^1_k\}_{k \in J_1}$ of open sets such that
\begin{enumerate}
    \item $\overline{W^1_k}$ are pairwise disjoint subsets of $U$ and
    \item if $f : X \to Y$ is a function such that $I_f$ encodes the collection $\{W^1_k\}_{k \in J_1}$ then it also encodes some collection $\{V_i^\prime\}_{i=1}^n \sqcup \{U_j^\prime\}_{j=1}^m$ with $V_i^\prime$ and $U_j^\prime$ as above.
\end{enumerate}
It follows from Corollary \ref{cor: encoding element recovers URPC} that if $f$ is as in item (ii), $I_f$ encodes a URPC cover for a $K_1$ and $\eps_1$.

Set 
\[
W^1 = \bigcup_{k \in J_1}W^1_k \subset U
\]
and observe that it is a marker set, since
\[
\sum_{i=1}^n\sqbrak{\, \overline{V_i^{-\delta}}\, } + \sum_{j=1}^m\sqbrak{\, \overline{U_j^{-\delta}}\, } \prec W^1.
\]
Thus, using the same argument, we can inductively construct collections $\{W^m_k\}_{k \in J_m}$ for $m \ge 2$ such that
\begin{enumerate}
    \item $\overline{W^m_k}$ are pairwise disjoint subsets of $W^{m-1}$ and
    \item if $I_f$ encodes the collection $\{W^m_k\}_{k \in J_m}$ then it also encodes a URPC cover for $K_m$ and $\eps_m$,
\end{enumerate}
where $W^m = \bigcup_{k \in J_m}W^m_k$. Without loss of generality, we may assume that any set $W^m_k$ lies in a single set from the collection $\{W^{m-1}_k\}_{k \in J_{m-1}}$. For consistent notation we also set $W^0_1 = W^0 = U$ and $J_0 = \{1\}$. By item (ii), to finish the proof it is sufficient to find a function $\phi : X \to [0, 1]^2$ supported on $U$ such that $\phi$ (and thus also $I_\phi$) encodes all of the collections $\{W^m_k\}_{k \in J_m}$.

Consider a rooted tree $T$ defined as follows. 
\begin{itemize}
    \item For all $m \ge 0$ there are $\abs{J_m}$ vertices on levels $2m$ and $2m+1$ which we denote by $\{v^0_{m, k}\}_{k \in J_m}$ and $\{v^1_{m, k}\}_{k \in J_m}$ respectively.
    \item Every vertex $v^0_{m, k}$ is connected to $v^1_{m, k}$ and these are the only edges between levels $2m$ and $2m+1$.
    \item A vertex $v^1_{m, k}$ on level $2m+1$ is connected to a vertex $v^0_{m+1, l}$ on level $2m+2$ if and only if the set $W^{m+1}_l$ is a subset of $W^m_k$.
\end{itemize}

The topological space $T \cup \partial T$ is a one-dimensional and compact, and there exists an embedding $\tau = (\tau_1, \tau_2) : T \cup \partial T \to [0, 1]^2$ which maps the root $v^0_{0, 1}$ to the point $(0, 0)$. We may additionally choose this embedding so that it satisfies 
\[
\tau_1(t) > \tau_2(t)
\]
for any point $t \in T \cup \partial T$ and so that $\tau(T \cup \partial T \setminus \{v^0_{0, 1}\}) \subset (0, 1)^2$.

For a vertex $v^0_{m, k}$ let $T^m_k$ be the subtree rooted at $v^0_{m, k}$. Naturally, $T^m_k \cup \partial T^m_k$ is a compact subset of $T \cup \partial T$.

Now choose a continuous function $\phi : X \to \tau\brak{T \cup \partial T} \subset [0, 1]^2$ such that 
\[
\phi(x) \in \tau\brak{T^m_k \cup \partial T^m_k \setminus \{v^0_{m, k}\}} \quad \Longleftrightarrow \quad x \in W^m_k.
\]
For instance, this function can be constructed as a limit of functions corresponding to finite parts of the tree. By construction, $\phi$ encodes every set $W^m_K$ and has the additional required properties by the choice of $\tau$.
\end{proof}

The lemma above is sufficient to deal with the case $M \ge 2$. A slight additional modification will allow us to also handle the case $M = 1$.

\begin{lemma}\label{lem: encoding URPC 1-dim}
Let $G \act X$ be an action with URPC and let $L \subset G$ be finite. Then there exists an open set $U \subset X$ such that
\[
\abs{Lx \cap \overline{U}} \le 2 \quad \mbox{for any} \ x \in X
\]
and a function $\phi : X \to [0, 1]$ supported on $U$ such that for any finite $K \subset G$ and $\eps > 0$ the map $I_\phi$ encodes a URPC cover for $K$ and $\eps$.
\end{lemma}

\begin{proof}
Without loss of generality, assume that $L$ contains the identity, and pick an arbitrary element $g \in G$. Set $L^\prime = L^{-1} \cup L^{-1}g$ and let $U^\prime$ be an $L^\prime$-free marker set. As usual, we can additionally assume that $\overline{U}$ is also $L^\prime$-free. By Lemma \ref{lem: encoding URPC 2-dim}, there is a function $\phi^\prime = (\phi_1, \phi_2) : X \to [0, 1]^2$ that encodes some URPC cover for any $K$ and $\eps$ and such that the open support of both $\phi_1$ and $\phi_2$ is equal to $U^\prime$ and $\phi_1 > \phi_2$ on $U^\prime$. Set $U = U^\prime \sqcup gU^\prime$ and define a function $\phi : X \to [0, 1]$ by
\[
\phi(x) = 
\begin{cases}
\phi_1(x) & \mbox{if} \ x \in U^\prime, \\
\phi_2(g^{-1}x) & \mbox{if} \ x \in gU^\prime, \\
0 & \mbox{if} \ x \in X \setminus \brak{U^\prime \sqcup gU^\prime}.
\end{cases}
\]
Clearly, we have
\[
\abs{Lx \cap \overline{U}} \le 2 \quad \mbox{for any} \ x \in X.
\]
Next, note that $x \in U^\prime$ if and only if $\phi(x) > \phi(gx) > 0$, hence $I_\phi$ encodes the set $U^\prime$. It follows that, by knowing $I_\phi$, it also possible to reconstruct the original function $\phi^\prime$ using the formula
\[
\phi^\prime(x) =
\begin{cases}
(0, 0) & \mbox{if} \ x \not \in U^\prime, \\
(\phi(x), \phi(gx)) & \mbox{if} \ x \in U^\prime.
\end{cases}
\]
Thus, $I_\phi$ encodes a URPC cover for any $K$ and $\eps$.
\end{proof}

\subsection{Mean dimension and shift embedding}

\begin{definition}
Let $\cU$ be an a collection of open subsets of a compact space $X$. The \emph{order} $\ord(\cU)$ of this collection is defined to be
\[
\ord(\cU) = \sup_{x \in X}\abs{\{U \in \cU \mid x \in U\}}.
\]
If $\cU$ is in fact an open cover of $X$, then we define the \emph{dimension} of the cover $\cU$ to be
\[
\dim(\cU) = \inf_{\cV}\ord(\cV),
\]
where $\cV$ ranges over all refinements of $\cU$. Note that the \emph{covering dimension} of the space $X$ is defined to be
\[
\dim(X) = \sup_{\cU}\dim(\cU),
\]
where $\cU$ ranges over all open covers of $X$.
\end{definition}

\begin{definition}
Let $G \act X$ be an action of an amenable group on a compact space and let $M \in \Rb_{\ge 0}$ be a nonnegative number. We say that $G \act X$ has \emph{mean dimension less than $M$} and write
\[
\mdim(G \act X) < M
\]
if there exists some $\gamma > 0$ so that for every open cover $\cU$ there exists a finite set $F \subset G$ and $\eps > 0$ such that for every $(F, \eps)$-invariant set $S \subset G$ we have
\[
\dim\brak{\bigvee_{g \in S}g^{-1}\cU} < (M-\gamma)\abs{S}.
\]
We define
\[
\mdim(G \act X) = \inf\brak{\{M \in \Rb_{\ge 0} \mid \mdim(G \act X) < M\}}.
\]
If the set above is empty we set $\mdim(G \act X) = \infty$.
\end{definition}

\begin{definition}
Let $\cU$ be an open cover of the space $X$ and let $\{\phi_U\}_{U \in \cU}$ be an associated partition of unity. We say that $\{\phi_U\}_{U \in \cU}$ is \emph{honest} if for any $U \in \cU$ we have 
\[
\overset{\circ}{\supp}(\phi_U) = U,
\]
where $\overset{\circ}{\supp}$ denotes the open support.

It is a basic exercise to show that any finite open cover has an associated honest partition of unity.
\end{definition}

\begin{remark}
Given a cover $\cU$ of the space $X$, the \emph{nerve complex} of $\cU$ is the simplicial complex with vertices $\{v_U\}_{U \in \cU}$ such that $v_{U_0}, v_{U_1}, \ldots, v_{U_k}$ span a simplex of dimension $k$ if and only if 
\[
U_0 \cap U_1 \cap \ldots \cap U_k \ne \emptyset.
\]
A partition of unity $\{\phi_U\}_{U \in \cU}$ then allows one to construct a simplicial approximation of the space $X$ by sending
\[
x \mapsto \sum_{U \in \cU}\phi_U(x)v_U
\]
in the nerve complex. If the partition of unity is honest, this simplicial approximation is \emph{essential}.

The language of simplicial approximations is heavily used in many prior works on mean dimension (for instance, \cite{GutLinTsu16}, \cite{GutQiaTsu19}). In this paper, however, we deliberately aim to avoid invoking simplicial approximations (since, in our opinion, it reduces the clarity of the argument) and formulate everything in terms of covers and partitions of unity. 
\end{remark}

We need the following easy lemma. The proof is left as an exercise to the reader.

\begin{lemma}\label{lem: honest partition}
Let $\cU$ be a finite open cover of the space, let $\cV$ be a finite collection of open sets, and let $\eps > 0$. Then there exists an honest partition of unity $\{\phi_U\}_{U \in \cU \sqcup \cV}$ for the cover $\cU \sqcup \cV$ such that
\[
\norm{\sum_{U \in \cV}\phi_U}< \eps.
\]
\end{lemma}

\begin{definition}
Let $N \in \Nb$ and let $I \subset \{1, 2, \ldots, N\}$. We naturally identify the cube $[0, 1]^N$ with $[0, 1]^{\{1, 2, \ldots, N\}}$ and define the usual projection $\pi_I : [0, 1]^N \to [0, 1]^I$ by
\[
\pi_I\brak{(x_i)_{i=1}^N} = (x_i)_{i \in I}.
\]

We always equip the cube $[0, 1]^I$ with the metric coming from the maximum norm
\[
\abs{(x_i)_{i \in I}}_\infty = \max_{i \in I}\abs{x_i}.
\]
When $Y$ is a topological space, the space of continuous function $F : Y \to [0, 1]^I$ is thus equipped with the metric coming from the uniform norm
\[
\norm{F} = \max_{y \in Y} \abs{F(y)}_\infty.
\]
\end{definition}

The next statement is a variation of \cite[Lemma 2.1]{GutTsu14}.

\begin{lemma}\label{lem: function perturbation}
Let $F : Y \to [0, 1]^N$ be a continuous function, let $\gamma \ge 0$ and  $\delta > 0$, let $\cU$ be a finite open cover of $Y$, and let $\cV$ be a finite collection of open sets such that
\begin{enumerate}
    \item $\ord(\cU) + \ord(\cV) \le (1-\gamma)N/2$,
    \item for any pair of points $x, y \in U$ in any $U \in \cU$ we have
    \[
    \abs{F(x) - F(y)}_\infty < \delta/2.
    \]
\end{enumerate}
Then there exists a function $\tilde{F} : Y \to [0, 1]^N$ such that 
\[
\norm{F - \tilde{F}} < \delta
\]
and for any subset $I \subset \{1, 2, \ldots, N\}$ with $\abs{I} > (1-\gamma)N$ the function $\pi_I(\tilde{F})$ encodes the cover $\cU \sqcup \cV$.
\end{lemma}

\begin{proof}
Use Lemma \ref{lem: honest partition} to find an honest partition of unity $\{\phi_U\}_{U \in \cU \sqcup \cV}$ for the cover $\cU \sqcup \cV$ such that
\begin{equation}\label{eq: partition of unity estimate}
\norm{\sum_{U \in \cV}\phi_U}< \frac{\delta}{4\norm{F}}.
\end{equation}
For each open set $U \in \cU \sqcup \cV$ choose an arbitrary point $y_U \in U$. First, consider the function $F^\prime : Y \to [0, 1]^N$ given by
\[
F^\prime(y) = \sum_{U \in \cU \sqcup \cV}\phi_U(y)F(y_U).
\]
It follows from condition (ii) and \eqref{eq: partition of unity estimate} that
\[
\norm{F - F^\prime} < \delta/2 + 2\norm{F}\frac{\delta}{4\norm{F}} = \delta.
\]
Denote
\[
\kappa = \norm{F - F^\prime} - \delta.
\]

For each $U \in \cU \sqcup \cV$ choose a vector $\xi_U \in B_\kappa(F(y_U)) \cap [0, 1]^N$ such that for any subset $I \subset \{1, 2, \ldots, N\}$ with $\abs{I} > (1-\gamma)N$ and any collection $\{U_1, U_2, \ldots, U_m\} \subset \cU \sqcup \cV$, where $m < (1-\gamma)N$, the vectors $\pi_I(\xi_{U_1}), \pi_I(\xi_{U_2}), \ldots, \pi_I(\xi_{U_m})$ are linearly independent (this can be achieved, for instance, by choosing each vector $\xi_U$ randomly from the set $B_\kappa(F(y_U)) \cap [0, 1]^N$). This, in particular, implies that for a vector $\xi \in [0, 1]^I$ there is at most one way to write it as a convex combination of less than $(1-\gamma)N/2$ vectors from the set $\{\pi_I(\xi_U)\}_{U \in \cU \sqcup \cV}$. Define
\[
\tilde{F}(y) = \sum_{U \in \cU \sqcup \cV}\phi_U(y)\xi_U. 
\]
Using condition (i), it follows from the discussion above that if $I \subset \{1, 2, \ldots, N\}$ with $\abs{I} > (1-\gamma)N$ then the value $\pi_I(\tilde{F}(y))$ determines all the coefficients $\phi_U(y)$. Since the partition of unity is honest, it follows that $\pi_I(\tilde{F})$ encodes the cover $\cU \sqcup \cV$. Additionally, by our choices we have 
\[
\norm{F - \tilde{F}} \le \norm{F - F^\prime} + \norm{F^\prime - \tilde{F}} <  \norm{F - F^\prime} + \kappa = \delta.
\]
\end{proof}

\begin{definition}
Recall that a map $F : (X, \rho) \to Y$ is called an \emph{$\eps$-embedding} if $F(x) \ne F(y)$ whenever $\rho(x, y) > \eps$.
\end{definition}

\begin{remark}\label{rem: eps embed open}
Let $\sC$ be a nonempty closed subset (hence a Baire space) in the space of continuous functions $C(X, [0, 1]^M)$. It is an easy exercise to show that the sets
\[
A_\eps = \{f \in \sC \mid I_f \ \mbox{is an} \ \eps\mbox{-embedding}\} 
\]
are open. If one can additionally show that they are also dense in $\sC$ then $A_\eps$ are comeager and therefore
\[
\bigcap_{n=1}^\infty A_{1/n} \ne \emptyset.
\]
Clearly, for any function $f$ in this intersection the map $I_f$ is a $G$-equivariant embedding of $X$ into $\brak{[0, 1]^M}^G$.
\end{remark}

\begin{theorem}\label{thm: embedding}
If $G \act X$ has URPC and 
\[
\mdim(G \act X) < M/2
\]
then there exists a $G$-equivariant embedding of $X$ into $\brak{[0, 1]^M}^G$.
\end{theorem}

\begin{proof}
It will be slightly more convenient for us to prove the result with the cube $[0, 1]^M$ replaced by $[-1, 1]^M$. Choose a small enough $\gamma > 0$ so that 
\[
\mdim(G \act X) < (1-\gamma)M/2.
\]
Let $L \subset G$ be a finite subset containing the identity such that $\abs{L} > 4/\gamma$. By Lemma \ref{lem: encoding URPC 1-dim}, there exists an open set $U$ and a function $\phi : X \to [0, 1]$ supported on it such that $I_\phi$ encodes some URPC cover for any $K$ and $\eps$ and 
\[
\abs{Lx \cap \overline{U}} \le 2
\]
for any $x \in X$. Define a closed subset of functions
\[
\sC = \{f \in C\brak{X, [-1, 1]^M} \mid \left.f\right|_{U} = (-\phi, 0, \ldots, 0) \ \mbox{and} \ \left.f\right|_{X \setminus U} \ge 0)\}.
\]
We immediately observe that for any $f = (f_1, f_2, \ldots, f_M) \in \sC$ we have
\[
-\phi(x) = \min(f_1(x), 0)
\]
and therefore if $I_\phi$ encodes a certain URPC cover then so does $I_f$ for any $f \in \sC$. By Remark \ref{rem: eps embed open}, it is sufficient to prove that the set of $f \in \sC$ such that $I_f$ is an $\eps$-embedding is dense.

Let $f \in \sC$ and let $\eps, \delta > 0$. In the rest of the proof we will construct a function $\tilde{f} \in \sC$ such that $I_{\tilde{f}}$ is an $\eps$-embedding and 
\[
\norm{f - \tilde{f}} < \delta.
\]
First, let $\cW$ be a finite open cover of $X$ such that for any $W \in \cW$ we have that $\diam(W) < \eps$ and $\abs{f(x) - f(y)}_\infty < \delta/2$ for any pair of points $x, y \in W$. Choose a finite set $K \subset G$ and an $\alpha > 0$ such that for any $(K, \alpha)$-invariant set $S \subset G$ we have that $\abs{Sx \cap \overline{U}} < \gamma\abs{S}/2$ for any $x \in X$ and
\[
\dim\brak{\bigvee_{g \in S}g^{-1}\cW} < \brak{(1-\gamma)M/2 - \ord(\cW)\alpha}\abs{S}.
\]

We now need to prepare some data and some notation. Let $\brak{\{S_i\}_{i=1}^n, \{V_i\}_{i=1}^n, \{g_j\}_{j=1}^m, \{U_j\}_{j=1}^m}$ be a URPC witness for $K$ and $\alpha$ such that $I_\phi$ encodes the corresponding cover $\{gV_i \colon 1 \le i \le n, g \in S_i\} \sqcup \{U_j\}_{j=1}^m$. As usual, we may assume that the sets $g\overline{V_i}$, for $1 \le i \le n$ and $g \in S_i$, are pairwise disjoint. Choose $\tau > 0$ small enough so that 
\begin{enumerate}
    \item $gV_i^\tau$ are pairwise disjoint for $1 \le i \le n$ and $g \in S_i$ and
    \item $\abs{S_ix \cap U^\tau} < \gamma\abs{S_i}/2$ for any $x \in X$ and any $1 \le i \le n$.
\end{enumerate}
Let $\psi : X \to [0, 1]$ be a bump function such that $\psi \equiv 1$ on $\overline{U}$ and $\supp\psi \subset U^\tau$ and let $\psi_i : X \to [0, 1]$ for $1 \le i \le n$ be bump functions such that $\psi_i \equiv 1$ on $\overline{V_i}$ and $\supp\psi \subset V_i^\tau$. We remark that the functions $\{g\psi_i \mid 1 \le i \le n, g \in S_i\}$ have pairwise disjoint supports and therefore the collection 
\[
\{g\psi_i \mid 1 \le i \le n, g \in S_i\} \sqcup \{1 - \sum_{i=1}^n\sum_{g \in S_i}g\psi_i\}
\]
is a partition of unity associated to the cover
\[
\{gV_i^\tau \mid 1 \le i \le n, g \in S_i\} \sqcup \{X \setminus \bigsqcup_{i=1}^nS_i\overline{V_i}\}.
\]
Finally, define a function $f^+ = (f^+_1, f^+_2, \ldots, f^+_M) : X \to [0, 1]^M$ by setting
\[
f^+_k(x) = \max(f_k(x), 0).
\]
Equivalently, $f^+ = f$ on $X \setminus U$ and $f^+ \equiv (0, 0, \ldots, 0)$ on $U$.

Next, fix an index $i \in \{1, 2, \ldots, n\}$. By our choices, there is a finite open cover $\cU_i$ of the set $V_i^\tau$ that refines $\left.\brak{\bigvee_{g \in S_i}g^{-1}\cW}\right|_{V_i^\tau}$ and such that
\[
\ord(\cU_i) < \brak{(1-\gamma)M/2 - \ord(\cW)\alpha}\abs{S_i}.
\]
By definition of a URPC witness, any set of the form $g_jU_j$, where $1 \le j \le m$ is a subset of one of the bases $V_1, V_2, \ldots, V_n$. Let
\[
J_i = \{1 \le j \le m \mid g_jU_j \subset V_i\}.
\]
Now define a collection of open subsets of $V_i$ by setting
\[
\cV_i = \bigsqcup_{j \in J_i} g_j\brak{\left.\cW\right|_{U_j}}. 
\]
Again, by definition of a URPC witness for $K$ and $\alpha$, we have that 
\[
\ord(\cV_i) \le \ord(\cW)\alpha\abs{S_i}
\]
and hence
\[
\ord(\cU_i) + \ord(\cV_i) < (1-\gamma)M\abs{S_i}/2.
\]
Consider a continuous function $F_i = (F_i^g)_{g \in S_i} : V_i^\tau \to \brak{[0, 1]^M}^{S_i}$ given by
\[
F_i^g(x) = f^+(gx).
\]
It is not difficult to see that $F_i$, $\cU_i$, and $\cV_i$ satisfy the assumptions of Lemma \ref{lem: function perturbation} and therefore there exists a function $\tilde{F}_i = (\tilde{F_i}^g)_{g \in S_i}: V_i^\tau \to \brak{[0, 1]^M}^{S_i}$ such that
\[
\norm{F_i - \tilde{F_i}} < \delta
\]
and $\pi_I(\tilde{F_i})$ encodes the cover $\cU_i \sqcup \cV_i$ for any subset $I \subset S_i$ with $\abs{I} > (1-\gamma)\abs{S_i}$ (here $\pi_I$ is a natural projection from $\brak{[0, 1]^M}^{S_i}$ to $\brak{[0, 1]^M}^I$). Now, first define
\[
\tilde{f}^+(x) = \brak{1 - \sum_{i=1}^n\sum_{g \in S_i}g\psi_i(x)}f^+(x) + \sum_{i=1}^n\sum_{g \in S_i}g\psi_i(x)\tilde{F_i}^g(g^{-1}x).
\]
By construction, $\tilde{f}^+$ is a function from $X$ to $[0, 1]^M$ such that
\[
\norm{f^+ - \tilde{f}^+} < \delta.
\]
Finally, let
\[
\tilde{f}(x) = \psi(x)f(x) + (1 -\psi(x))\tilde{f}^+(x).
\]
We have that $\tilde{f} = f = (-\phi, 0, \ldots, 0)$ on $U$, hence $
\tilde{f} \in \sC$, and that
\[
\norm{f - \tilde{f}} < \delta.
\]

It remains to show that $I_{\tilde{f}}$ is an $\eps$-embedding. Indeed, let $x, y \in X$ such that $\rho(x, y)>\eps$. Since $\tilde{f} \in \sC$, the map $I_{\tilde{f}}$ encodes the URPC cover $\{gV_i \colon 1 \le i \le n, g \in S_i\} \sqcup \{U_j\}_{j=1}^m$. Thus, if $x \in U$ and $y \not \in U$ for some element $U$ of the cover then $I_{\tilde{f}}(x) \ne I_{\tilde{f}}(y)$ and we are done. Suppose now that both $x$ and $y$ belong to the same element $U$ in the URPC cover. It is either of the form $gV_i$ for $g \in S_i$, in which case we set $h = g^{-1}$, or of the form $U_j$, in which case we set $h = g_j$. In either case, both $hx$ and $hy$ are elements in some base $V_i$. Since all elements in $\cW$ have diameter at most $\eps$, there is at least one element of $\cW$ that separates $x$ and $y$. It follows that there is then at least one element of $\cU_i \sqcup \cV_i$ that separates $hx$ and $hy$. Next, define
\[
I_x = \{g \in S_i \mid ghx \not \in U^\tau\} \quad \mbox{and} \quad I_y = \{g \in S_i \mid ghy \not \in U^\tau\}.
\]
By construction, $\abs{I_x}, \abs{I_y} > (1 - \gamma/2)\abs{S_i}$ and therefore the set $I = I_x \cap I_y$ satisfies
\[
\abs{I} > (1-\gamma)\abs{S_i}.
\]
We have 
\[
(\tilde{f}(ghx))_{g \in I} = \pi_I(\tilde{F}_i)(hx) \ne \pi_I(\tilde{F}_i)(hy) = (\tilde{f}(ghy))_{g \in I}
\]
since $\pi_I(\tilde{F}_i)$ encodes the cover $\cU_i \sqcup \cV_i$. Hence $I_{\tilde{f}}(hx) \ne I_{\tilde{f}}(hy)$ and therefore $I_{\tilde{f}}(x) \ne I_{\tilde{f}}(y)$ by equivariance.

\end{proof}

\section{Shift embeddability for amenable actions of nonamenable groups}
\label{sec: nonamen shift embed}

\begin{definition}\label{def: topologically amenable}
An action $\GonX$ of a discrete countable group on a compact metrizable space is said to be \emph{topologically amenable} if there is a sequence of maps $\mu_n : X \to \mathrm{Prob}(G)$ such that
\[
\sup_{x \in X}\norm{g\mu_n(x) - \mu_n(gx)}_1 \to 0
\]
for any $g \in G$, where $\mathrm{Prob}(G)$ is the space of probability measures on $G$ equipped with the $l^1$-norm.
\end{definition}

The following is \cite[Theorem 3.9]{GarGefKraNar23}\footnote{There it was proved for $D = \{e, h\}$. The proof for arbitrary $D$ is exactly the same.}.

\begin{theorem}\label{thm: GGKN}
Let $F_2 < G$ be a nonamenable group containing a free subgroup on two generators. Let $\GonX$ be a topologically amenable action. Then for any finite set $D \subset F_2$ there are open sets $V_1, V_2, V_3 \subset X$ and elements $g_1, g_2, g_3 \in X$ such that 
\begin{itemize}
    \item the sets $g\overline{V_i}$, where $g \in D$ and $i =1, 2, 3$, are pairwise disjoint and
    \item the sets $g_iV_i$, $i = 1, 2, 3$, cover $X$.
\end{itemize}
\end{theorem}

\begin{corollary}
Suppose $F_2 < G$ contains a free subgroup on two generators and let $\GonX$ be an action that satisfies the conclusion of Theorem \ref{thm: GGKN}. Then for any finite set $D \subset F_2$ there is a $D$-free marker set (in other words, the restricted action $F_2 \act X$ has the marker property).
\end{corollary}

An argument similar to the one in Subsection \ref{ssec: encoding URPC} shows the following.

\begin{lemma}\label{lem: encoding nonamen}
Suppose $F_2 < G$ contains a free subgroup on two generators, let $\GonX$ be an action that satisfies the conclusion of Theorem \ref{thm: GGKN}, and let $L \subset F_2$ be finite. Then there exists an open set $U \subset X$ and a function $\phi : X \to [0, 1]$ supported on $U$ such that
\begin{itemize}
    \item $Lx \cap \overline{U} \le 2$ for any $x \in X$ and
    \item for any finite set $D \subset F_2$ the map $I_\phi$ encodes sets $V_1, V_2, V_3$ that satisfy the conclusion of Theorem \ref{thm: GGKN}.
\end{itemize}
\end{lemma}

We now prove the main theorem of this section. The proof is an easier version of Theorem \ref{thm: embedding}.

\begin{theorem}\label{thm: embedding nonamen}
Suppose $F_2 < G$ contains a free subgroup on two generators and let $\GonX$ be an action that satisfies the conclusion of Theorem \ref{thm: GGKN}. Then there is a $G$-equivariant embedding of $X$ into $[0, 1]^G$.
\end{theorem}

\begin{proof}
Once again, we will prove the statement with the cube $[0, 1]^G$ replaced by $[-1, 1]^G$. Let $L \subset F_2$ be any finite set with $\abs{L} > 5$. Let $U$ and $\phi$ be given by Lemma \ref{lem: encoding nonamen} and consider the function space
\[
\sC = \{f \in C(X, [-1, 1]) \mid \left.f\right|_U = -\phi \ \mbox{and} \ \left.f\right|_{X \setminus U} \ge 0\}.
\]
We claim that the set of functions $f$ such that $I_f$ is an $\eps$-embeddings is dense in $\sC$. Indeed, suppose $f \in \sC$ and $\eps, \delta > 0$ are given. We will construct a function $\tilde{f} \in \sC$ such that $I_{\tilde{f}}$ is an $\eps$-embedding and
\[
\norm{f - \tilde{f}} < \delta.
\]

To do that, let $\cW$ be a finite open cover consisting of sets with diameter less than $\eps$ and such that for any $W \in \cW$ and any pair of points $x, y \in W$ we have that
\[
\abs{f(x) - f(y)}_\infty < \delta/2.
\]
Choose a finite set $D \subset F_2$ of cardinality $\abs{D} > 10\ord(\cW)$ that is a disjoint union of translates of $L$. Then for any $x \in X$ we have
\[
\abs{Dx \setminus \overline{U}} > \frac{3}{5}\abs{D}.
\]
By our choice of $\phi$, there are open sets $V_1, V_2, V_3$ encoded by $I_\phi$ and $g_1, g_2, g_3 \in G$ that satisfy the conclusion of Theorem \ref{thm: GGKN} with the set $D$ chosen above. Choose $\tau > 0$ small enough so that 
\begin{enumerate}
    \item $\abs{Dx \setminus U^\tau} > 3\abs{D}/5$ for any $x \in X$ and
    \item the sets $gV_i^\tau$, $g \in D$, $i = 1, 2, 3$ are pairwise disjoint.
\end{enumerate}
Let $\psi : X \to [0, 1]$ be a bump function such that $\psi \equiv 1$ on $U$ and $\supp \psi \subset U^\tau$ and let $\psi_i$ for $i = 1, 2, 3$ be bump functions such that $\psi_i \equiv 1$ on $V_i$ and $\supp \psi_i \subset V_i^\tau$.

Define 
\[
f^+(x) = \max(f(x), 0)
\]
and for each $i = 1, 2, 3$ let $F_i = (F_i^g)_{g \in D} : V_i^\tau \to [0, 1]^D$ be defined by
\[
F_i^g(x) = f^+(gx).
\]
Define a cover $\cU_i$ of the set $V_i^\tau$ by
\[
\cU_i = g_i^{-1}\brak{\left.\cW\right|_{g_iV_i}}.
\]
Then $\ord(\cU_i) \le \ord(\cW)$, hence, by Lemma \ref{lem: function perturbation}, there is a function $\tilde{F} : V_i^\tau \to [0, 1]^D$ such that 
\[
\norm{F - \tilde{F}} < \delta
\]
and $\pi_I(\tilde{F})$ encodes $\cU_i$ for any subset $I \subset D$ with $\abs{I} > 2\ord(\cU_i)$. Set 
\[
\tilde{f}^+(x) = \brak{1 - \sum_{i=1}^3\sum_{g \in D}g\psi_i(x)}f^+(x) + \sum_{i=1}^3\sum_{g \in D}g\psi_i(x)\tilde{F_i}^g(g^{-1}x)
\]
and
\[
\tilde{f}(x) = \psi(x)f(x) + (1 -\psi(x))\tilde{f}^+(x).
\]
One can easily check that
\[
\norm{f - \tilde{f}} < \delta.
\]

Now, suppose that $x, y \in X$ and $\rho(x, y) > \eps$. Recall that $\{g_iV_i\}_{i=1}^3$ is a cover of the space encoded by $I_\phi$. Hence, if $x$ and $y$ belong to different elements of this cover then $I_\phi(x) \ne I_\phi(y)$ and therefore $I_{\tilde{f}}(x) \ne I_{\tilde{f}}(y)$. Assume now that they lie in the same element, say $g_iV_i$. Since $\cW$ was a cover with diameters smaller than $\eps$, these points lie in different elements of $\cW$ and hence $g_i^{-1}x$ and $g_i^{-1}y$ belong to a different elements in $\cU_i$. Let 
\[
I_x = \{g \in D \mid gg_i^{-1}x \not \in U^\tau\} \quad \mbox{and} \quad I_y = \{g \in D \mid gg_i^{-1}y \not \in U^\tau\}
\]
and define $I = I_x \cap I_y$. By our choices, $\abs{I} > \abs{D}/5 > 2\ord(\cU_i)$, hence $\pi_I(\tilde{F})$ encodes the cover $\cU_i$. It follows that
\[
(\tilde{f}(gg_i^{-1}x))_{g \in I} = \pi_I(\tilde{F}_i)(g_i^{-1}x) \ne \pi_I(\tilde{F}_i)(g_i^{-1}y) = (\tilde{f}(gg_i^{-1}y))_{g \in I}
\]
and thus 
\[
I_{\tilde{f}}(x) \ne I_{\tilde{f}}(y).
\]
\end{proof}


\begin{thebibliography}{999}


\bibitem{Ber23}
A. Bernshteyn.
Distributed algorithms, the Lovász Local Lemma, and descriptive combinatorics.
{\it Invent. math.} {\bf 233} (2023), 495-–542.



\bibitem{CooKri05}
M. Coornaert and F. Krieger. 
Mean topological dimension for actions of discrete amenable groups. 
{\it Discrete Contin. Dyn. Syst.}, {\bf 13}(3) (2005), 779--793.

\bibitem{DowHucZha19}
T. Downarowicz, D. Huczek, and G. Zhang.
Tilings of amenable groups.
{\it J. Reine Angew. Math.} {\bf 747} (2019), 277--298.

\bibitem{DowZha23}
T. Downarowicz and G. Zhang. 
Symbolic extensions of amenable group actions and the comparison property.
{\it Mem. Am. Math. Soc.} {\bf 281} (2023).

\bibitem{EllNiu17}
G. A. Elliott and Z. Niu.
The C$^*$-algebra of a minimal homeomorphism of zero mean dimension. 
{\it Duke Math.\ J.}\ {\bf 166} (2017), 3569--3594.


\bibitem{GarGefKraNar23}
E. Gardella, S. Geffen, J. Kranz, P. Naryshkin.
Classifiability of crossed products by nonamenable groups. 
{\it J. Reine Angew. Math.} {\bf 797} (2023), 285–-312.




\bibitem{Gro99}
M. Gromov. 
Topological invariants of dynamical systems and spaces of holomorphic maps. {\it I. Math. Phys. Anal. Geom.}, {\bf 2} (1999), 323--415.

\bibitem{Gut15}
Y. Gutman.
Mean dimension \& Jaworski-type theorems. 
{\it Proc. Lond. Math. Soc.} {\bf 111} (2015), no. 4, 831–-850.

\bibitem{Gut17}
Y. Gutman.
Embedding topological dynamical systems with periodic points in cubical shifts. 
{\it Ergod. Theory Dyn. Syst.} {\bf 37} (2017), 512-–538.

\bibitem{GutLinTsu16}
Y. Gutman, E. Lindenstrauss, and M. Tsukamoto.
Mean dimension of $\Zb^k$-actions.
{\it Geom. and Func. An. } {\bf 26} (2016), 778--817.

\bibitem{GutQiaTsu19}
Y. Gutman, Y. Qiao, and M. Tsukamoto.
Application of signal analysis to the embedding problem of $\Zb^k$-actions.
{\it Geom. and Func. An.} {\bf 29} (2019), 1440-–1502.

\bibitem{GutTsu14}
Y. Gutman and M. Tsukamoto.
Mean dimension and a sharp embedding theorem: extensions of aperiodic subshifts.
{\it Ergod. Theory Dyn. Syst.} {\bf 34} (2014), 1888–-1896.

\bibitem{GutTsu20}
Y. Gutman and M. Tsukamoto.
Embedding minimal dynamical systems into Hilbert cubes.
{\it Invent. Math.} {\bf 221} (2020), 113–-166.


\bibitem{JinParQia22}
L. Jin, K. K. Park, Y. Qiao.
Mean dimension and a non-embeddable example for amenable group actions.
{\it Fund. Math.} {\bf 257} (2022), 19--38.

\bibitem{Ker20}
D. Kerr. 
Dimension, comparison, and almost finiteness.
{\it J. Eur.\ Math.\ Soc.}\ {\bf 22} (2020), 3697--3745. 

\bibitem{KerNar21}
D. Kerr and P. Naryshkin.
Elementary amenability and almost finiteness.
arXiv:2107.05273

\bibitem{KerSza20}
D. Kerr and G. Szab{\'o}. 
Almost finiteness and the small boundary property.
{\it Comm.\ Math.\ Phys.}\ {\bf 374} (2020), 1--31.

\bibitem{LanSza23}
E. Lanckriet, G. Szabó.
On embeddings of extensions of almost finite actions into cubical shifts. {\it Colloq. Math.} {\bf 174} (2023), 229--240.

\bibitem{Lev23}
M. Levin.
Finite-to-one equivariant maps and mean dimension.
arXiv:2312.04689

\bibitem{Li13}
H. Li.
Sofic mean dimension. 
{\it Adv. Math.} {\bf 244} (2013), 570--604.

\bibitem{LiNiu20}
C. G. Li and Z. Niu.
Stable rank of $C(X) \rtimes \Gamma$.
arXiv:2008.03361



\bibitem{Lin99}
E. Lindenstrauss.
Mean dimension, small entropy factors and an embedding theorem. 
{\it Inst. Hautes Études Sci. Publ. Math.}, {\bf 89}(1) (1999), 227--262.

\bibitem{LinTsu14}
E. Lindenstrauss and M. Tsukamoto. 
Mean dimension and an embedding problem: an example. 
{\it Israel J. Math}, {\bf 199} (2014), 573--584.

\bibitem{LinWei00}
E. Lindenstrauss and B. Weiss. 
Mean topological dimension. 
{\it Israel J. Math.}, {\bf 115} (2000), 1--24.

\bibitem{Ma21}
X. Ma.
A generalized type semigroup and dynamical comparison.
{\it Ergod. Theory Dyn. Syst.} {\bf 41} (2021), 2148--2165.

\bibitem{Nar22}
P. Naryshkin.
Polynomial growth, comparison, and the small boundary property.
\textit{Adv. Math.} \textbf{406} (2022).

\bibitem{Nar24}
P. Naryshkin.
Group extensions preserve almost finiteness.
{\it J. Funct. Anal.} {\bf 286} (2024), 110348.

\bibitem{NarPet24}
P. Naryshkin and S. Petrakos.
$\mathcal{Z}$-stability of crossed products by topological full groups.
arXiv:2401.06006, {\it updated version in preparation}

\bibitem{Niu22}
Z. Niu.
Comparison radius and mean topological dimension: 
Rokhlin property, comparison of open sets, and subhomogeneous C$^*$-algebras.
{\it J. Anal.\ Math.} {\bf 146} (2022), 595--672.

\bibitem{Niu23}
Z. Niu.
Comparison radius and mean topological dimension: $\Zb^d$-actions.
{\it Canad. J. Math.} (2023), 1--27.

\bibitem{Niu21}
Z. Niu.
$\cZ$-stability of $C(X) \rtimes \Gamma$. 
{\it Trans. Amer. Math. Soc.}, {\bf 374} (2021), 7525--7551.


\bibitem{Shi23}
R. Shi.
Finite mean dimension and marker property.
{\it Trans. Amer. Math. Soc.} {\bf 376} (2023), 6123--6139.

\bibitem{TomWin13}
A.~Toms and W.~Winter.
Minimal dynamics and K-theoretic rigidity: Elliott's conjecture,
{\it Geom. and Func. An.} {\bf 23} (2013), 467--481.


\bibitem{TsuTsuYos22}
M. Tsukamoto, M. Tsutaya, and M. Yoshinaga. 
$G$-index, topological dynamics and the marker property. 
{\it Israel J. Math} {\bf 251} (2022), 737-–764.


\bibitem{WeiHe23}
S. Wei and Z. He.
A generalization of topological Rokhlin dimension and an embedding result.
arXiv:2311.05721

\end{thebibliography}
\end{document}